\documentclass[reqno,11pt]{amsart}
\usepackage[english]{babel}
\usepackage{amssymb,verbatim,microtype,graphicx,hyperref,cite}
\usepackage[T1]{fontenc}
\usepackage{xcolor}

\usepackage{amsmath}
\usepackage{amssymb}

 \newcommand{\beq}{\begin{equation}}
\newcommand{\eeq}{\end{equation}}

\newtheorem{theorem}{Theorem}[section]
\newtheorem{proposition}[theorem]{Proposition}
\newtheorem{lemma}[theorem]{Lemma}
\newtheorem{corollary}[theorem]{Corollary}

\newtheorem{definition}[theorem]{Definition}
\newtheorem{example}[theorem]{Example}

\newtheorem{remark}[theorem]{Remark}

\numberwithin{equation}{section}

\newcommand{\PSH}[1]{\mbox{$\mathcal{PSH}(#1)$}}
\newcommand{\LPSH}[1]{\mbox{$\mathfrak{L}(#1)$}}
\newcommand{\LPSHs}[1]{\mbox{$\mathfrak{L^*}(#1)$}}
\newcommand{\TPSH}[1]{\mbox{$\mathcal{TPSH}(#1)$}}
\newcommand{\TLPSH}[1]{\mbox{$\mathcal{T}\mathfrak{L}(#1)$}}
\newcommand{\TLPSHs}[1]{\mbox{$\mathcal{T}\mathfrak{L^*}(#1)$}}
\newcommand{\LCT}{{\operatorname{c_\infty}}}
\newcommand{\LCTh}{{\operatorname{\hat c_\infty}}}

\newcommand{\D}{{\mathbb D}}
\newcommand{\Z}{{\mathbb Z}}
\newcommand{\R}{{\mathbb R}}
\newcommand{\B}{{\mathbb B}}
\newcommand{\Tn}{{\mathbb T^n}}
\newcommand{\Rn}{{\mathbb R}^n}
\newcommand{\Rnm}{{\mathbb R}_-^n}
\newcommand{\Rnp}{{\mathbb R}_+^n}
\newcommand{\C}{{\mathbb  C}}

\newcommand{\Cn}{{\mathbb  C\sp n}}

\newcommand{\cH}{{\mathcal H}}
\newcommand{\cI}{{\mathcal  I}}
\newcommand{\cJ}{{\mathcal  J}}
\newcommand{\cO}{{\mathcal  O}}
\newcommand{\cP}{{\mathcal P_\infty}}
\newcommand{\wcP}{{\mathcal P}}
\newcommand{\cPs}{{\mathcal  P^*}}

\newcommand{\bone}{{\bf 1}}
\newcommand{\simk}{{\C^n\setminus K}}

\newcommand{\LL}{\mathcal{L}}
\def\Sp{\textnormal{Sp}}
\def\Lin{\mathbf L}

\def\geq{\geqslant}
\def\leq{\leqslant}
\def\m{\mathrm m}
\def\supp {\mathrm{supp}}

\def\to{\longrightarrow}

\begin{document}

\title{Log canonical thresholds at infinity}

\author{Carles Bivià-Ausina}\address{Institut Universitari de Matemàtica Pura i Aplicada, Universitat Politècnica de València, Camí de Vera s/n, 46022 València, Spain}\email{carbivia@mat.upv.es}

\author{Alexander Rashkovskii}\address{Department of Mathematics and Physics, University of Stavanger, 4036 Stavanger, Norway}\email{alexander.rashkovskii@uis.no}

\keywords{Plurisubharmonic function of logarithmic growth, polynomial map, log canonical threshold, Newton polytope}
\subjclass[2020]{32U05, 32U15, 32W20, 32A08}

\date{\today}

\begin{abstract} The paper considers a global version of the notion of log canonical threshold $\LCT(u)$ for plurisubharmonic functions $u$ of logarithmic growth in $\Cn$, aiming at description of the range of all $p>0$ such that $e^{-u}\in L^p(\Cn)$. Explicit formulas for $\LCT(u)$ are obtained in the toric case. By considering Bergman functions of corresponding weighted Hilbert spaces, a new polynomial approximation of
plurisubharmonic functions of logarithmic growth with control over
its singularities and behavior at infinity (a global version of Demailly's approximation theorem) is established. Some applications to polynomial maps are given. 
\end{abstract}

\maketitle

\section{Introduction}

Behavior of a plurisubharmonic function $u$ near its singularity points is of fundamental interest in pluripotential theory, and such characteristics of plurisubharmonic singularities as Lelong numbers and their generalizations are standard tools of complex analysts and geometers. Another powerful approach concerns local integrability properties of $e^{-u}$. If $u(\zeta)\neq-\infty$, then $e^{-u}\in L^p$ near the point $\zeta$ for any $p>0$, while this need not be true if $u(\zeta)=-\infty$. Traditionally, one measures the singularity of $u$ at $\zeta$ by means of the {\it log canonical threshold} $c_\zeta(u)$ at $\zeta$ (the {\sl least upper bound} of the set of all positive numbers $c$ such that $e^{-cu}$ is square-integrable near $\zeta$) or, equivalently, by its reciprocal $\lambda_\zeta(u)=1/c_\zeta(u)$, the {\it integrability index} at $\zeta$. After the fundamental paper \cite{DK}, it was realized that these notions play a huge role in local pluripotential theory and its applications, which has resulted in great activity in the area. There is a vast literature on the subject, here we mention only a few publications: \cite{ACKPZ}, \cite{B}, \cite{BF1}, \cite{BL}, \cite{DP}, \cite{Ein2004}, \cite{dFEM1},  \cite{Gue}, \cite{Kollar}, \cite{McZ}, \cite{Mu2}, \cite{Rash13}-\cite{Rash17}, and many more; we deliberately omit here the whole range of papers concerning the openness conjecture and related topics because they are a bit out of our focus here. 

For functions plurisubharmonic in the whole space $\Cn$, it is also highly desirable to measure their asymptotic behavior near infinity. Of special interest are functions of logarithmic growth: $u(z)\le \sigma_u\log|z|+O(1)$ for $|z|\gg 1$; we denote the class of such functions by $\LPSH\Cn$. 
The main motivation for studying it is its relevance to polynomial maps.

Our goal here is to consider global version(s) of the log canonical threshold (and integrability index)  for this class of functions, which would allow us to describe the range of all $p>0$ such that $e^{-u}\in L^p(\Cn)$. We were not able to find any detailed study of this subject in the literature.

We introduce a notion of {\it log canonical threshold at infinity}, $\LCT(u)$, as the {\sl greatest lower bound} of the set of all positive numbers $c$ such that $e^{-cu}$ is square-integrable near infinity (that is, outside a compact subset of $\Cn$). Its reciprocal $\lambda_\infty(u)=1/\LCT(u)$ is the {\it integrability index of $u$ at infinity}.

  While for $n=1$, $\lambda_\infty(u)$ coincides with the logarithmic type $\sigma_u$ of $u$ (Theorem~\ref{thm:C}), the situation for $n>1$ is much more complicated. In particular, we do not know if the (square-)integrability of $e^{-cu}$ near infinity implies that for $e^{-c'u}$ with   $c'>c$: note that $\liminf u(z)$ as $|z|\to\infty$ can be equal to $-\infty$.  It is much simpler to work with the value $\LCTh(u):=\LCT(\max\{u,0\})$. Evidently, $\LCTh(u)\le\LCT(u)$, however there exist polynomials $P$ in $\C^2$ with $\LCTh(\log|P|)<\LCT(\log|P|)=\infty$, see Example~\ref{ex:P2}. 
  
  The two values evidently coincide for functions from the important 
class $\LPSHs\Cn$ of plurisubharmonic functions $u$ of logarithmic growth  that are exhaustive at infinity (i.e., $u(z)\to+\infty$  as $|z|\to \infty$). For such functions, we establish some useful elementary properties of $\LCT(u)$, parallel to those of $c_a(u)$. In particular, $\LCT(u)\ge\LCT(\Psi_u)$ with equality in the case of toric $u$; here $\Psi_u$ is the logarithmic indicator of $u$ introduced in \cite{Rash01}. The value of $\LCT(\Psi_u)$ can be computed by analogs of Kiselman's and Howald's formulas (Proposition~\ref{prop:ind_Kiselman}), in terms of the indicator diagram of $u$, which in the polynomial case $u=\log|P|$ is the Newton polyhedron of the map $P$ at infinity. The formulas extend to arbitrary {\sl toric} functions from the class $\LPSHs\Cn$ (Theorem~\ref{thm:toric_Kiselman}) and and to
polynomial maps that are Newton non-degenerate at infinity (Corollary~\ref{cor:NNDlct}).

We also consider {\it multipliers at infinity} of $u\in\LPSH\Cn $, i.e., polynomials $P$ such that $|P|e^{-u}$ is square-integrable at infinity. While multipliers improve integrability of $e^{-u}$ near finite points, at infinity they worsen it; the collection $\cP(u)$ of such polynomials measures the `integrability surplus' for $e^{-u}$ at infinity. It is a finite-dimensional vector space with the property $P\in\cP(u)$, {\sl provided $|P|\le |Q|$ for some $Q\in\cP(u)$}; contrary to the local situation, $\cP(u)$ is not an ideal. In the case of toric $u\in\LPSHs\Cn$, the space $\cP(u)$ has monomial generators and can be described explicitly in terms of the indicator diagram of $u$ (Theorem~\ref{theo:How2}), which corresponds to Howald's theorem in the local case.

Even more interesting is the vector space $\wcP(u)$ of polynomials $P$ such that $|P|e^{-u}\in L^2(\Cn)$. Been endowed with the inner product
\[
\langle P,Q \rangle_u = \int_\Cn P \overline{Q} e^{-2u}\,\beta_n,
\]
it becomes a Hilbert space, and one can consider the corresponding Bergman kernel function. In the local situation of functions in a bounded domain, the celebrated Demailly's theorem uses such a function for approximation of $u$ by functions with analytic singularities, converging to $u$ with control over the singularities. In the global case, we modify a construction from \cite{Berman} and consider the Bergman kernel functions $B_m$ of the Hilbert spaces $\wcP(mu+\kappa\Lambda)$ with $\Lambda(z) =\frac12\log(1+|z|^2)\in\LPSHs\Cn$, $\kappa>n$ fixed, and $m>0$. In Theorem~\ref{thm:Deminf}, we establish convergence of $u_m=\frac1{2m}\log B_{m}$ to $u$, point-wise and in $L_{loc}^1$, that controls the local singularities and the behavior at infinity (including, if $\liminf(z)/\log|z|>0$ as $|z|\to \infty$, the log canonical threshold at infinity). 

Note that in \cite{Berman} (as well as in other related papers), the corresponding Hilbert spaces were formed by polynomials $P$ satisfying, in addition to the integrability condition, the restriction $\deg P< m$. This resulted in convergence of the corresponding functions $u_m$ to the largest psh function not exceeding $\min\{\phi,\Lambda + O(1)\}$, and not to $u$; for more details, see Remark~\ref{rem:Berman}.

We also look for possible connections between the log canonical threshold at infinity and the Monge-Ampère measure, in the spirit of a relation between $c_\zeta(u)$ and the residual Monge-Ampère mass $(dd^cu)^n(\zeta)$ from \cite{D09} or a stronger inequality from \cite{DP}. 
It turns out that no universal relation between $\LCT(u)$ and the total Monge-Ampère mass $(dd^cu)^n(\Cn)$ exists: for any $N>1$, there exist toric functions $u_N,\, v_N\in\LPSHs\Cn$ such that $\LCT(u_N)=\LCT(v_N)=1$ while $(dd^cu)^n(\Cn)=2/N$ and $(dd^cv)^n(\Cn)=2N$ (Example~\ref{ex:antiD}). On the other hand, we prove a bound similar to the one from \cite{D09} for toric functions whose indicator diagrams at infinity are so-called {\it lower sets} (Proposition~\ref{prop:FMDglob}).

In this paper, we do not treat advanced subjects like the openness property, semicontinuity or restriction formula that are quite central in the local theory of log canonical thresholds. We keep this for the future, considering the topics discussed here just as first steps in the area not studied yet and that has a strong potential.  

\medskip
The presentation is organized as follows. In Section~2, we recall some preliminaries on plurisubharmonic functions of logarithmic growth and polynomial maps, toric functions and indicators. We also prove there the equality between the {\it global multiplicity} of a polynomial map $P$ with finite zero set and the total Monge-Ampère mass of $\log|P|$. In Section~3, we define the log canonical thresholds $\LCT(u)$ and $\LCTh(u)$ at infinity and present some instructive examples. In Section~4, we give a complete description of the one-dimensional case. Some elementary properties of $\LCT(u)$ are discussed in Section~5. In Section~6, we derive formulas for computation of the log canonical threshold at infinity for logarithmic indicators, and we extend them to toric plurisubharmonic functions and monomial maps (and, more generally, polynomial maps that are Newton non-degenerate at infinity) in Section~7. In Section~8, we introduce Hilbert spaces $\cP(u)$ and $\wcP(u)$ of polynomial multipliers, and we give their description for the case of toric $u$. In Section~9, we consider Bergman functions of such Hilbert spaces for establishing a new approximation of arbitrary plurisubharmonic functions of logarithmic growth by logarihms of moduli of polynomial maps, with control over its singularities and behavior at infinity (a global version Demailly's approximation theorem). In Section~10, we study relations between the log canonical threshold $\LCT(u)$ and the total Monge-Ampère mass $(dd^cu)^n(\Cn)$.

The following general notation and conventions are used:

$|z|=\left(\sum_k|z_k|^2\right)^{1/2}$ is the Euclidean norm in $\Cn$ (or $\Rn$);

$\Cn\ni z\to\infty$ means $|z|\to\infty$;

$\Rnp=\{x=(x_1,\ldots,x_n)\in\Rn:\: x_k>0, \ 1\le k\le n\}$;

$\D_R^n=\{z\in\Cn:\: |z_k|<R, \ 1\le k\le n\}$, $\D^n=\D_1^n$;


$\beta_n$ is the Euclidean volume form in $\Cn$;

$d=\partial+\bar\partial$, $d^c=(\partial-\bar\partial)/2\pi i$;

$a^+=\max\{a,0\}$; in particular, $\log^+t=\max\{\log t, 0\}$;

$\inf\emptyset=+\infty$, $\sup\emptyset=-\infty$.

\medskip{\bf Acknowledgments.} The second-named author thanks Alexander Eremenko for clarifying the subject in dimension $1$. Part of the work on the project was carried on during the first-named author's visit to the Department of Mathematics and Physics of the University of Stavanger (April 2024) and the second-named author's visits to the Institut Universitari de Matemàtica Pura i Aplicada of the Universitat Politècnica de València (November 2023 and November 2024). The authors are grateful to the host institutions for their hospitality. 

\section{Preliminaries}

\subsection{Plurisubharmonic functions of logarithmic growth and polynomial mappings}\label{sect:pshpol}

Here we recall some standard facts on plurisubharmonic functions and there relations to polynomial maps. 
For general properties of plurisubharmonic functions and the complex Monge-Ampère operators, we refer to  \cite{Dbook}, \cite{GZ}, \cite{Klimek}. There is a vast literature on functions of logarithmic growth, starting from works of Lelong, Siciak and Zakhariuta; for the recent state of the theory, see, for example, \cite{BBL}, \cite{MSSS} and references therein. The only new fact in this subsection is Proposition~\ref{prop:globDem} relating the total mass $(dd^c \log|P|)^n(\Cn)$ for a polynomial map $P$ to its global multiplicity considered in \cite{BH}.

\medskip

Let $\LPSH\Cn$ denote the class of all plurisubharmonic (psh) functions $u$ of logarithmic growth in $\Cn$: $u(z)= O(\log|z|)$  as $|z|\to\infty$;
in other words, a psh function $u$ belongs to $\LPSH\Cn$ if its {\it logarithmic type} $\sigma_u$ is finite:
\beq\label{eq:logtype}
\sigma_u:=\limsup_{z\to\infty} u(z)/\log|z|<\infty;
\eeq
equivalently,
\[
\sigma_u=\left(dd^cu\wedge(dd^c\log|z-a|)^{n-1}\right)(\Cn).
\]
One always has the bound
\beq\label{eq:logtype1}
u(z)\le\sigma_u\log^+|z|+C
\eeq
with some $C=C(u)\in\R$. (Note that the {\it Lelong class} $\mathcal L$ is the collection of all $u\in \LPSH\Cn$ with $\sigma_u\le 1$.)

We will also consider the subclass $\LPSHs\Cn$ of functions $u\in \LPSH\Cn$ such that $u(z)\to +\infty$ as $|z|\to\infty$; equivalently, $\{z: u(z)<C\}\Subset\Cn$ for any $C\in\R$, meaning that $u$ is {\it exhaustive on $\Cn$}.

We recall that the Lelong number of a psh function $u$ at a point $\zeta$ is 
\[
\nu_\zeta(u)=\liminf_{z\to \zeta}\frac{u(z)}{\log|z-\zeta|}\ge 0
\]
or, equivalently, 
\[
\nu_\zeta(u)=\left(dd^cu\wedge(dd^c\log|z-\zeta|)^{n-1}\right)(\{\zeta\}),
\]
so the logarithmic type can be viewed as the `Lelong number at infinity'.

A point $\zeta\in\Cn$ is a {\it singularity point} of $u$ if the $u$ is not bounded in any neighborhood of $\zeta$, a typical example being $\zeta\in u^{-1}(-\infty)$. If $\zeta$ is an isolated singularity point of $u$, then the complex Monge-Ampère operators $(dd^cu)^k$, $k\le n$, are well defined near $\zeta$, and the {\it residual Monge-Ampère mass} 
\beq\label{eq:massLNzeta}
\tau_\zeta(u):=(dd^cu)^n(\{\zeta\})\ge \nu_\zeta(u)^n.
\eeq 
If $(dd^cu)^n$ is well defined in the whole $\Cn$ (in particular, if $u\in\LPSHs\Cn$), then the {\it total Monge-Ampère mass} 
\beq\label{eq:massLNinfty}
\m_n(u):=\int_\Cn (dd^c u)^n\le \sigma_u^n.
\eeq

\medskip

The class $\LPSH\Cn$ is closed under addition and taking maximum, and it is a psh counterpart for the polynomial ring $\C[z_1,\ldots,z_n]$, model examples of $u\in\LPSH\Cn$ being $u=c\log|P|$, $c>0$, for polynomials $P$ and, more generally, polynomial maps $P:\Cn\to \C^q$. Note that, by a theorem of Siciak \cite{Si96} (see also Theorem~2.9 in Appendix~B of \cite{STot}), any $u\in\LPSH\Cn$ is the limit of a decreasing sequence of functions 
$$u_k=\max_{1\le j\le t_k} c_{jk} \log|P_{jk}|$$
for some polynomials $P_{jk}$ and constants  $0<c_{jk}\le \sigma_u/\deg(P_{jk})$. The collection of all polynomial maps $P:\Cn\to\C^q$, $q=1,2,\ldots$, will be denoted by $\wcP$, and the condition $P\in\cPs$ will mean $\log|P|\in\LPSHs\Cn$; in other words, $\cPs$ is the collection of all {\it tame} polynomial maps.

Singularity points of the psh functions $\log|P|$ are zeros of the maps $P\in\wcP$, and characteristics of their singularities are classical objects of complex analysis and commutative algebra. For example, the Lelong number $\nu_a(\log|P|)$ is the minimum of the vanishing orders (multiplicities) of the components $P_j$ of $P$, and the logarithmic type $\sigma_{\log|P|}$ is the maximum of their degrees. 

We recall that the {\it Łojasiewicz exponent at infinity} $\LL_\infty(P)$ of a polynomial map 
\(P:\mathbb{C}^{n}\rightarrow \mathbb{C}^{q}\) is the greatest real number \(l \) such that the inequality \(|P(z)|\ge C|z|^{l }\) holds for some constant \(C>0\) and all large enough \(|z|\) (see for instance \cite{Krasinski}). We extend this notion to psh functions of logarithmic growth by calling the value 
   \beq\label{eq:Le}
    \LL_\infty(u)=\liminf_{z\to\infty}\frac{u(z)}{\log|z|}
    \eeq
    {\it the  Łojasiewicz exponent of $u\in\LPSH\Cn$ at $\infty$}. Evidently, $\LL_\infty(u)\ge 0$ for any $u\in\LPSHs\Cn$, and $\LL_\infty(u)> 0$ implies $u\in\LPSHs\Cn$.

\medskip

If $\zeta$ is an isolated zero of a polynomial (or, more generally, holomorphic) map $P:\Cn\to\C^q$, $q\ge n$, then 
the residual Monge-Ampère mass $\tau_\zeta(\log|P|)$ equals the Samuel multiplicity $\m_\zeta(P)$ of the local ideal $\cI_\zeta(P)$ generated by $P_1,\ldots, P_q$ at $\zeta$ \cite[Lemma 2.1]{D09}. 
In particular, $\m_\zeta(P)$ equals the generic value of the multiplicity of the equidimensional maps $P^A:\Cn\to\C^n$ whose components $P_k^A$ are linear combinations $\sum_jA_{jk}P_j$ with $A_{jk}$ entries of constant matrices $A\in M_{p\times n}(\C)$. Note that the current $(dd^c\log|P^A|)^n$ equals the intersection current $dd^c\log|P_1^A|\wedge\ldots\wedge dd^c\log|P_n^A|$, and its mass $\m_\zeta(P^A)$ at $\zeta$ is the degree of the map $P^A$ at $\zeta$.

For a version of {\it global multiplicity} of polynomial maps, one can consider the following notion. 

\begin{definition}\label{def:globmult}{\rm \cite{BH}}
Let $P:\Cn\to\C^q$ be a polynomial map such that $P^{-1}(0)$ is finite. Its {\it global multiplicity} $\m(P)$
is the largest possible value of $\mu(P^A)$ with the maps $P^A$ as above, whenever $(P^A)^{-1}(0)$ is finite. Here, for any equidimensional polynomial map $F$ with finite $F^{-1}(0)$,
\[
\mu(F)=\dim\frac{\C[z_1,\ldots,z_n]}{\cI(F)}=\sum_{\zeta\in F^{-1}(0)}\dim\frac{\cO_\zeta}{\cI_\zeta(F)}.
\]
where $\cI(F)$ denotes the ideal of $\C[z_1,\ldots,z_n]$ generated by the components of $F$. 
\end{definition} 
For the second equality here, see e.g. \cite[p. 150]{CLS}. Note that the zero set $P^{-1}(0)$ might be empty. 

\medskip

As in the local case, we can express the global multiplicity $\m(P)$ in terms of the  Monge-Ampère measure of $\log|P|$:

\begin{proposition}\label{prop:globDem}
    For any polynomial map $P:\Cn\to\C^q$ with finite $P^{-1}(0)$, we have
\beq\label{eq:Demglobal}
\m(P)=(dd^c\log|P|)^n(\Cn).
\eeq
\end{proposition}

\begin{proof}
When $q=n$, this follows from King's formula 
\[
(dd^c\log|P|)^n=\sum_{\zeta\in P^{-1}(0)} \m_\zeta(P)\,\delta_\zeta.
\]
Let $q>n$. For any point $\alpha$ in the Grassmannian $G_{q,n}$ of $n$-dimensional subspaces in $\C^q$, consider the function $u_\alpha=\log|P^A|$, where $P^A$ is the equidimensional map as above with $A$ the matrix of an orthonormal basis of $\alpha$. For all $\alpha$ in a Zariski open set of $G_{q,n}$, the set $P^{-1}_A(0)$ is finite \cite[Theorem~13.5.1]{SW}, and the Crofton type formula (2.2) of \cite{D09}
\[
(dd^c\log|P(z)|)^n=\int_{G_{q,n}} (dd^c u_\alpha(z))^n\,d\mu(\alpha),
\]
$\mu$ being the Haar measure on $G_{q,n}$, is valid for the same reasons as in \cite[Lemma 2.1]{D09}. Since $P^A$ are equidimensional, we get
\[
(dd^c\log|P|)^n(\Cn)=\int_{G_{q,n}} \m(P^A)\,d\mu(\alpha).
\]
By the continuity of the integer-valued function $\m(P^A)$, it is constant on a Zariski open set of $G_{q,n}$ and equals $\m(P)$, which gives us \eqref{eq:Demglobal}.
\end{proof}

\subsection{Logarithmic indicators and Newton diagrams}\label{ssec:ind}

 A substantial tool of this work is the notion of logarithmic indicator introduced in \cite{Rash01}. Here we recall its definition and basic properties, see  \cite{Rash01} --\cite{Rash13}. 

 \begin{definition} Given a function $u\in \LPSH{\Cn}$, its {\it logarithmic indicator} is 
\beq\label{eq:ind}
\Psi_u(z)=\lim_{R\to\infty}R^{-1}\sup\{u(\zeta):\: |\zeta_k|\le |z_k|^R, 
\ 1\le k\le n\}
\eeq
for $z\in\Cn$ with $z_1\cdot z_2\cdot\ldots z_n\ne 0$, and extended to the whole $\Cn$ by continuity. 
\end{definition}

We have $\Psi_u\in\LPSH\Cn$ and
\beq\label{eq:ind_bound}
u(z)\le\Psi_u(z)+ C, \quad z\in\Cn,
\eeq
see \cite[Theorem 2.4]{Rash01}; in particular, $\Psi_u$ belongs to $\LPSHs\Cn$ if $u$ does.

The main properties of $\Psi_u$ are
\beq\label{eq:ind_hom}
\Psi_u(z)=\Psi_u(|z_1|,\ldots|z_n|),\quad 
\Psi_u(|z|_1^c,\ldots|z_n|^c)=c\,\Psi_u(z) \ \forall c>0.
\eeq

This implies that the function 
\beq\label{eq:psiu}
\psi_u(t):=\Psi_u(e^{t_1},\ldots, e^{t_n}), \quad t\in\Rn,
\eeq
is a convex, positive homogeneous function on $\Rn$, increasing in each $t_k$; we will call it the {\it convex image} of $\Psi_u$. Therefore, it is the supporting function of a compact convex set $\Gamma_u\subset\overline\Rnp$:
\beq\label{eq:indsupp}
\psi_u(t)=h_{\Gamma_u}(t), \quad t\in\Rn;
\eeq
recall that the supporting function $h_K$ of a convex set $K\subset\Rn$ is 
\[
h_K(t)(t)=\sup\{\langle a,t\rangle:\: a\in K\}.
\]
Explicitly, 
\beq\label{eq:indicdiag}
\Gamma_u=\{a\in\Rn: \langle a,t\rangle \le \psi_u(t) \quad \forall t\in\Rn\}. 
\eeq
Since $\Psi_u(z)=\psi_u(\log|z_1|,\ldots,\log|z_n|)$, we call it the {\it logarithmic supporting function} of $\Gamma_u$, and the set $\Gamma_u$ is called the {\it indicator diagram of $u$}. 

We will also consider the set $\Gamma_{u,\infty}$, {\it the indicator diagram of $u$ of at infinity},  
\beq\label{eq:indicdiaginfty}
\Gamma_{u,\infty}=\{a\in\Rn: \langle a,t\rangle \le \psi_u^+(t) \quad \forall t\in\Rn\},
\eeq
so that
\[
\psi_u^+(t)=h_{\Gamma_{u,\infty}}(t), \quad t\in\Rn.
\]

The corresponding local object is the indicator $\Psi_{u,\zeta}(z)$ which is the restriction of the logarithmic indicator of the function $u(z+\zeta)$ to the unit polydisk $\D^n$. Its convex image $\psi_{u,\zeta}$ is the supporting function of the {\it indicator diagram $\Gamma_{u,\zeta}$ of $u$ at $\zeta$}, i.e., a (unique) closed convex set $\Gamma\subset\overline\Rnp$ with the property $\Gamma+\Rnp\subset\Gamma$ and such that  $\psi_{u,\zeta}|_{\Rnm}$ is (the restriction to $\Rnm$ of) the supporting function of $\Gamma$ \cite{Rash00}:
   \beq\label{eq:G0}
   \Gamma_{u,\zeta}=\{a\in\Rn: \langle a,t\rangle \le \psi_{u,\zeta}(t) \quad \forall t\in\Rnm\}. 
\eeq

\medskip

In the polynomial case, the above indicator diagrams are classical objects \cite{AYu}, \cite{Ku}. 
Namely, the indicator diagram  $\Gamma_{\log|P|}$ of a polynomial $P=\sum_J c_Jz^J$  is the {\it Newton polyhedron (polytope)} $\Gamma_P$ of $P$, that is, the convex hull of the set  $\supp(P)=\{J:\: c_J\neq 0\}$, the indicator diagram  $\Gamma_{\log|P|,\infty}$ is the {\it Newton polyhedron $\Gamma_{P,\infty}$ of $P$ at infinity} (or the {\it global Newton polyhedron}), i.e., the convex hull of the set $\supp(P)\cup\{0\}$, and $\Gamma_{\log|P|,0}$ (we have chosen, for simplicity, $\zeta=0$) is the {\it Newton polyhedron $\Gamma_{P,0}$ of $P$ at $0$}, the convex hull of the set 
\[
\bigcup_{J\in \supp(P)}(J+\Rnp).
\]
For a polynomial map $P=(P_1,\ldots,P_q)$, the Newton polyhedra $\Gamma_{P}$, $\Gamma_{P,\infty}$ and  $\Gamma_{P,0}$  are the convex hulls of the unions of the corresponding Newton polyhedra of the component functions of $P$.

\medskip

The indicator diagrams can be used for efficient estimates of the Monge-Ampère masses $\tau_\zeta(u)$ \cite{Rash00} and $\m_n(u)$ \cite{Rash01}, which are refined versions of \eqref{eq:massLNzeta} and \eqref{eq:massLNinfty}:
for any $u\in\LPSHs\Cn$, 
\beq\label{eq:multvola}
\tau_u(\zeta)\ge \tau_{\Psi_{u,\zeta}}(0)=n!\,{\rm Vol}(\Rnp\setminus\Gamma_{u,\zeta})
\eeq
and
\beq\label{eq:multvolglob}
\m_n(u)\le \m_n(\Psi_{u})=n!\,{\rm Vol}(\Gamma_{u,\infty}).
\eeq
For $u=\log|P|$, they give known bounds of the multiplicities of polynomial maps $P\in\cPs$ in terms of (co)volumes of the Newton polyhedra. 

\section{Log canonical thresholds: definitions and examples}

Recall that the {\it log canonical threshold}, or the {\it complex singularity exponent}, at a point $\zeta\in\Cn$ of a function $u$ plurisubharmonic in a neighborhood of $\zeta$ is 
\[
c_\zeta(u)=\sup\{c>0:\: e^{-cu}\in L^2(\zeta)\},
\]
where the condition $g\in L^2(\zeta)$ means that there exists a neighborhood $U$ of $\zeta$ such that 
\[
\int_{U} |g|^2\,\beta_n<\infty.
\]
Its reciprocal $\lambda_\zeta(u)=1/c_\zeta(u)$ is called the {\it integrability index} of $u$ at $\zeta$; see, e.g., \cite{Kiselman}, \cite{DK}. 

These values are related to the Lelong number of $u$ at $\zeta$ 
by {\it Skoda's inequalities} \cite{Sk}
\beq\label{eq:Skoda}
\frac{\nu_\zeta(u)}{n} \le \lambda_\zeta(u) \le {\nu_\zeta(u)}, \quad 
\frac1{\nu_\zeta(u)} \le c_\zeta(u) \le \frac{n}{\nu_\zeta(u)}.
\eeq
In particular, when $n=1$, we have $\lambda_\zeta(u)=\nu_\zeta(u)$ and $c_\zeta(u)=1/{\nu_\zeta(u)}$.

\begin{definition}\label{def:LCT}
The log canonical threshold of $u\in\LPSH\Cn$ at infinity is
\beq\label{eq:LCT}  
\LCT(u):=\inf\{c>0:\: e^{-cu}\in L^2(\infty)\},
\eeq
and its integrability index at $\infty$ is
\[
\lambda_u(\infty):=1/\LCT(u)=\sup\{\lambda>0:\: e^{-u/\lambda}\in L^2(\infty)\}.
\]
\end{definition}

Here, the condition $g\in L^2(\infty)$ means that there exists a compact set $K\subset\Cn$ such that 
\[
\int_{\Cn\setminus K} |g|^2\,\beta_n<\infty.
\]
For $u=\log|P|$, the integrability condition rewrites as
\[
\int_{\simk} |P|^{-2c}\,\beta_n<\infty,
\]
and we will use the denotation $\LCT(P)$ for $\LCT(\log|P|)$ in this case. 

\begin{remark}
Note that the conditions $c_\zeta(u)=\infty$ and $\LCT(u)=\infty$ have opposite meanings: while the former means $e^{-cu}\in L^2(\zeta)$ for all $c>0$, the latter is $e^{-cu}\not\in L^2(\infty)$ for all $c>0$.
\end{remark}

\begin{remark}\label{rem:ray}
      Since $u$ can be negative somewhere near $\infty$, it is not evident that the set of $c>0$ such that $e^{-cu}\in L^2(\infty)$ is a ray or empty (which is obvious if $u$ is positive on a neighborhood of $\infty$, but not in general). 
\end{remark}

To avoid the issue described in Remark~\ref{rem:ray}, one can consider 
\[
\LCTh(u):=\LCT(u^+)=\inf\{c>0:\: u^{-cu^+}\in L^2(\Cn)\}.
\]
This value is actually easier to work with than $\LCT(u)$; in particular, it is clear that the set $\{c>0:\: u^{-cu^+}\in L^2(\Cn)\}$ is a ray or empty.
We also set $\hat\lambda_\infty(u)=1/\LCTh(u)$.

Of course, $\LCTh(u)\le \LCT(u)$
for any $u\in\LPSH\Cn$, and $\LCTh(u)=\LCT(u)$ if $u$ is bounded from below near $\infty$. 
It turns out that the equality need not hold for arbitrary $u\in\LPSH\Cn$, even in the polynomial case, see Example \ref{ex:P2} below.
\medskip

When $n=1$, the case $u=\log|P|$ is quite straightforward, while for arbitrary subharmonic functions $u$ of logarithmic growth it is less trivial but still manageable; see the next section. 
When $n>1$, the situation is much more complicated even in the polynomial case. Take, for example, $P(z)=z_1$; then, by Fubini's theorem, for any $c\ge 0$,
\[
\int_{\Cn} \max\{|z_1|,1\}^{-2c}\,\beta_n=+\infty,
\]
so $\LCTh(P)=\LCT(P)=+\infty$. Similarly, $\LCTh(u)=+\infty$ if $u$ is independent of some of the variables. Furthermore,  this remains true for any monomial $z^J$: assuming $J_1\neq0$ and using polar coordinates for each $z_k$, we get, for any $c\ge 0$,  
\begin{eqnarray*}
  \int_{\C^{n}} \max\{|z^J|,1\}^{-2c}\,\beta_n & \ge & 
  \int_{\{|z^J|<1\}} \,\beta_n = 
  (2\pi)^n \int_{\{|r^J|<1\}} \prod_k r_k\, dr \\
  & = & 2^{n-1}\pi^n \prod_{k=2}^n \int_0^\infty r_k^{1-2J_k/J_1}\,dr_k=+\infty.
  \end{eqnarray*}

Of course, we have $\LCTh(u)=+\infty$ for every function $u$ such that the volume of some its sublevel set $\{u<C\}$ is infinite. As the above computations show, the monomial functions $u=\log|z^J|$ have this property. It turns out that this need not be true for general polynomials. 

\begin{example}\label{ex:P2}
    Let $P_k=z_1^k z_2(z_2^2-1)$, $z\in\C^2$; denote $\rho(z_2)= |z_2(z_2^2-1)|^{-1/k}$. Then 
    \beq\label{eq:boundvol}
        {\rm Vol}(\{|P_k|\le 1\}) =2\pi \int_\C \int_0^{\rho(z_2)}\, r\,dr \,\beta_1(z_2)
        =\pi \int_\C |z_2(z_2^2-1)|^{-2/k}.
    \eeq
    The last integral converges if $k=2$, while it diverges near $0$ and $\pm 1$ for $k=1$, and it diverges at infinity for all $k\ge 3$. We have thus established ${\rm Vol}(\{|P_k|\le 1\})=+\infty$ and, therefore,  $\LCTh(P_k)=+\infty$, for all $k\neq 2$, while ${\rm Vol}(\{|P_2|\le 1\})<+\infty$.

Furthermore, for any $c>1/2$,
\begin{eqnarray*}
    \int_{\{|P_2|>1\}} |P_2|^{-2c}\,\beta_n &=& 
  2\pi \int_{\C}|z_2(z_2^2-1)|^{-2c}\,\beta_1(z_2)\int_{\rho(z_2)}^\infty r^{1-4c}\, dr \\
  & = & \frac\pi{2c-1} \int_\C |z_2(z_2^2-1)|^{-1}\,\beta_1(z_2)<+\infty.
  \end{eqnarray*}
   Together with \eqref{eq:boundvol} this gives us, for any $c>1/2$,
   \[
   \int_{\C^{2}} \max\{|P_2|,1\}^{-2c}\,\beta_2= 
   \frac{2c\pi}{2c-1} \int_\C |z_2(z_2^2-1)|^{-1}\,\beta_1(z_2)<+\infty
   \]
   and so, $\LCTh(P_2)=\frac12$.

On the other hand, one can see that $\LCT(P)=+\infty$
for any polynomial 
$P$ which is monomial in one of the variables. Indeed, let $P(z_1,z')=z_1^kQ(z')$. 
Then we get, for any polydisk $\D_R^n$ and any $c\ge 0$,  
\[
\int_{\C^{n}\setminus \D_R^n} |P|^{-2c}\,\beta_n \ge 
\int_{\C^{n-1}\setminus \D_R^{n-1}} |Q(z')|^{-2c}\beta_{n-1}(z') \int_{\C} |z_1^k|^{-2c} \,\beta_1(z_1) =
+\infty. 
\]
In particular, for the aforementioned function $P_2$, we get $$\frac12=\LCTh(P_2)<\LCT(P_2)=+\infty.$$
\end{example}

\medskip

\begin{remark}
We do not know if, when $n>1$, 
\beq\label{eq:wholespace}
\int_{\simk} |P|^{-2c}\,\beta_n=\infty\quad \forall c\in\R
\eeq
for any polynomial $P$ (and, more generally, any polynomial map to $\C^q$, $q<n$) and any compact $K\subset\Cn$.
\end{remark}

In addition, as the next example shows, one can have $\LCTh(P)=\infty$ even for nondegenerate polynomial maps $P:\Cn\to\Cn$.

\begin{example}\label{ex:nondominating}
    Let $P:\C^2\to\C^2$ be defined as $P(z)=(z_1,z_1z_2-1)$. Since $\log|P|$ is locally bounded, $|P|^{-2c}$ is integrable on any bounded domain. We have $|P|< 2$ on the unbounded set $E=\{z:\: |z_1|<1,\ |z_1z_2-1|<1\}$. It is easy to see that the set $E$ has infinite volume:
    \begin{eqnarray*}
        {\rm Vol}(E) &=& \int_{\{z_1:\,|z_1|<1\}}
    \int_{\{z_2:\,|z_2-|z_1|^{-1}|<|z_1|^{-1}\}}\beta_1(z_2)\,\beta_1(z_1) \\
    &= & \pi \int_{\{z_1:\,|z_1|<1\}} |z_1|^{-2}\,\beta_1(z_1)
    = 2\pi^2\int_0^{1} r^{-1}\,dr=\infty.
    \end{eqnarray*}
    Therefore, for any $c>0$, 
    \[
    \int_{\C^2} \max\{|P|,2\}^{-2c}\beta_2\ge 2^{-4c} \,{\rm Vol}(E) =\infty.
    \]
\end{example}

\begin{remark}
    One could also consider convergence with respect to the weighted volume form $\beta_n^\kappa:=(|z|^2+1)^{-\kappa}\,\beta_n$ for any $\kappa>n$. Since the $\beta_n^\kappa$-measure of $\Cn$ is finite, we get then the corresponding log canonical threshold $c_{\kappa,\infty}(u^+)=0$ for any $u\in\LPSH\Cn$. In some cases, however, the measure $\beta_n^\kappa$ is quite useful, see Section~\ref{sec:appr}.
\end{remark}

\section{One-dimensional case}

For any polynomial $P:\C\to\C$, we have $|P(z)|\asymp |z|^{\deg P}$ (up to a multiplicative constant) for $|z|\gg 1$, so 
\[
\LCT(P)=\LCTh(P)=1/\deg P.
\]
The same is true for all polynomial maps $P=(P_1,\ldots,P_p): \C\to\C^p$. The case of arbitrary $u\in\LPSH\C$ is more subtle.

\medskip
Recall that for any function subharmonic in a domain $\Omega\subset\C$, its {\it Riesz measure} is $\mu_u :=\frac1{2\pi}\Delta u\ge 0$. The total Riesz mass $\mu_u(\C)$ of any $u\in\LPSH\C$ is finite; moreover,
\[
\mu_u(\C)=\sigma_u,
\]
see, e.g., \cite{HCJ}. Furthermore, any $u\in\LPSH\C$ has representation
\beq\label{eq:repr}
u(z)=C+\int_{D_R} \log|z-w|\,d\mu_u(w)+\int_{\C\setminus D_R} \log\left|1-\frac{z}{w}\right|\,d\mu_u(w), \ z\in\C,
\eeq
with any $R>0$.
If $\mu_u$ decreases at infinity so that 
$\int_{\C\setminus D_R} |\log |w||\,d\mu_u<\infty$, one has a simpler formula,
\beq\label{eq:repr1}
u(z)=C+\int_{\C} \log|z-w|\,d\mu_u(w), \quad z\in\C.
\eeq

The Riesz measure $\mu_u$ determines also other main characteristics of behavior of $u$. In particular, its charge $\mu_u(\zeta)$ at $\zeta\in\C$ equals the Lelong number $\nu_u(\zeta)$. 
By Skoda's inequalities \eqref{eq:Skoda}, we have 
\beq\label{eq:Skoda1}
c_\zeta(u)=1/\mu_u({\zeta}).
\eeq
For functions $u\in\LPSH\C$, the corresponding relation \beq\label{eq:n1} 
\LCT(u)=1/\sigma_u
\eeq
can be deduced from Cartan's lemma (\cite[Lemma 6.17]{HCJ} or Theorem 3 in Section 11.2 of \cite{Levin}) stating that logarithmic potentials \eqref{eq:repr1} are asymptotically close to $\sigma_u\log|z|$, outside an exceptional set. Since Cartan's lemma is quite an advanced tool, we present here a more elementary proof of \eqref{eq:n1} due to Alexander Eremenko (private communication).

\begin{theorem}\label{thm:C}
   For any $u\in\LPSH\C$, $\LCT(u)=\LCTh(u)=1/\sigma_u$. Moreover, for any nonconstant $u\in\LPSH\C$, 
   \[\label{eq:convn=1}
   \{c>0:\: e^{-cu}\in L^2(\infty)\}=\{c>0:\: e^{-cu^+}\in L^2(\infty)\}=(1/\sigma_u,+\infty).
   \]
\end{theorem}

\begin{proof}
    Assume $\sigma_u>0$ (otherwise $u$ is constant and the equalities are obvious). Then 
   $    \LCT(u)\ge\LCTh(u)\ge 1/\sigma_u $
    because $u^+\le \sigma_u\log|z|+O(1)$ as $z\to \infty$ and $\LCT(\log|z|)=\LCTh(\log|z|)=1$, so we need to prove $e^{-cu}\in L^2(\infty)$ for any $c>1/\sigma_u$ but not for $c=1/\sigma_u$. The latter statement is trivial since $u\le \sigma_u \log^+|z|+C$ and $e^{- \log|z|^+}\not\in L^2(\infty)$, so we concentrate here on proving the former. 
    
    It suffices to prove the following claim: {\sl If $\sigma_u>2$, then there exists $R>0$ such that} 
    \[
    \int_{|z|>R} e^{-u}\beta_1<\infty.
    \]
    
    \medskip
    In the representation \eqref{eq:repr}, we can assume $C=0$. 
    If the support of the Riesz measure $\mu_u$ is contained in some disk $D_R$, then from \eqref{eq:repr} we get for all $z$ with $|z|>2R$, 
    \beq\label{eq:J1}
    u(z)\ge \sigma_u \log(|z|-R) \ge \sigma_u(\log|z| - \log2)
    \eeq
and hence, as $\sigma_u>2$,
\[
    \int_{|z|>2R} e^{-u}\beta_1 \le 
    \frac{2^{1+\sigma_u}}\pi  \int_{2R}^\infty r^{1-\sigma_u}\,dr<\infty.
    \]

In the general case, choose $R>0$ such that $s_1:=\mu_u(D_R)>2$ and $s_2:=\mu_u(\C\setminus D_R)<1$. In representation \eqref{eq:repr}, denote the first integral by $J_1(z)$ and the second one by $J_2(z)$. Then we have, as in \eqref{eq:J1},
\beq\label{eq:J12}
J_1(z)\ge s_1(\log|z| - \log2)
    \eeq
for all $z$ with $|z|>2R$.

Denote by $\tilde\mu$ the restriction of the measure $s_2^{-2}\mu_u$ to $\C\setminus D_R$. since $\tilde\mu$ is a probability measure on $\C\setminus D_R$, Jensen's inequality gives us
\[
e^{-J_2(z)}= \exp\left\{-\int_{\C\setminus D_R} s_2\log\left|1-\frac{z}{w}\right|d\tilde\mu(w)\right\}
\le \int_{\C\setminus D_R} \left(\frac{|w|}{|w-z|}
\right)^{s_2}d\tilde\mu(w).
\]
  Using Fubini's theorem and the bound \eqref{eq:J12}, we get
  \begin{eqnarray*}
  \int_{\C\setminus D_{2R}} e^{-u(z)}\beta_1(z)
  &\le & \int_{\C\setminus D_R} \int_{\C\setminus D_{2R}}  \left(\frac{|w|}{|w-z|}
\right)^{s_2} e^{-J_1(z)}\beta_1(z) d\tilde\mu(w)
\\
 &\le & \int_{\C\setminus D_R} \int_{\C\setminus D_{2R}}  \left(\frac{|w|}{|w-z|}
\right)^{s_2} 2^{s_1} |z|^{-s_1}\beta_1(z) d\tilde\mu(w).
  \end{eqnarray*}
  
  For any fixed $w\in \C\setminus D_R$, the inner integral, $I(w)$, splits into the some of the  integral $I_1(w)$ inside the disk $\{z: |z-w|<\frac12 |w|\}$ and $I_2(w)$ outside it. We have
  \[
  I_1(w)\le 2\pi |w|^{s_2} \int_0^{\frac12|w|} 
  r^{1-s_2}|w|^{-s_1} dr
  \le C_1 |w|^{2-s_1}\le C_1 R^{2-s_1}
  \]
(recall that $s_1>2$) and
\[
I_2(w)\le \int_{\C\setminus D_{2R}} 2^{s_2-s_1} |z|^{-s_1}\beta_1(z)= C_2 R^{2-s_1}.
\]
 Therefore,
 \[
 \int_{\C\setminus D_{2R}} e^{-u(z)}\beta_1(z) \le \int_{\C\setminus D_R} I(w) d\tilde\mu(w) 
 \le (C_1+C_2) R^{2-s_1} <\infty,
 \]
which completes the proof. 
\end{proof}

\begin{remark}\label{rem:ray1}
   The fact that the set $\{c>0:\: e^{-cu}\in L^2(\infty)\}$ is a ray, is not obvious; cf. Remark~\ref{rem:ray}.  \end{remark}

As a consequence of Theorem~\ref{thm:C}, we have

\begin{corollary}
    A function $u\in\LPSH\C$ satisfies 
    \beq\label{eq:allc}
    \int_\C e^{-2cu}\beta_1=\infty \quad \forall c>0
    \eeq
    if and only if $u(z)=s+t\log|z-\zeta|$ for some $s\in\R$, $t\ge 0$ and $\zeta\in\C$.
\end{corollary}
\begin{proof}
    The `if' direction being trivial, let us prove the reverse. Let a non-constant $u\in\LPSH\C$ satisfy \eqref{eq:allc}. For each $c>0$ and any $R>0$, we have
    \[
    \int_\C e^{-2cu}\beta_1=\int_{|z|<R} e^{-2cu}\beta_1 + \int_{|z|>R} e^{-2cu}\beta_1 =I_1(c,R)+I_2(c,R) =\infty.
    \]
    For any $c\le \sigma_u^{-1}$ and all $R>0$, we have $I_2(c,R)=\infty$. If $c> \sigma_u^{-1}$, we can find $R'>R$ such that $I_2(c,R')<\infty$. Then $I_1(c,R')$ must be infinite, which implies $c> M_u^{-1}(R')$, where $M_u(R)=\sup\{\mu_u(\zeta): |\zeta|\le R\}$. 
    
    We have $\sigma_u=\mu_u(\C)\ge M_u(R)$ for any $\zeta\in \C$, the equality being possible only if $u$ is harmonic outside the point $\zeta$, in which case $u(z)=s+\sigma_u\log|z-\zeta|$ for some $s\in\R$. Otherwise, $\sigma_u>M_u(R)$, so there exists $c'\in (\sigma_u^{-1}, M_u^{-1}(R))$ and hence, both $I_1(c',R')$ and $I_2(c',R') $ are finite, which contradicts \eqref{eq:allc}. 
    \end{proof}

Another simple consequence, in view of \eqref{eq:Skoda1}, is 

\begin{corollary}\label{cor:locglob1}
    For any $u\in\LPSH\C$, $\LCT(u)\le \inf
    \{c_\zeta(u):\: \zeta\in\C\}$.
\end{corollary}

\section{Elementary properties and comparison to other characteristics}

\medskip

As we saw in the previous section, dealing with functions not increasing to infinity requires additional efforts. In the one-dimension case, this was achieved by using their representation as potentials. In several variables, we do not have such a tool, so handling of such `bad' functions is more complicated. We leave this problem for a future treatment, concentrating here mainly on functions $u\in\LPSHs\Cn$ and polynomial maps $P\in\cPs$ (i.e., functions tending to $+\infty$ at infinity).
In particular, 
$ \LCTh(u)=\LCT(u)$ for all $u\in\LPSHs\Cn$. This -- as well as some other results below -- is also true for functions bounded from below near infinity, however we will not explore this issue here.

\medskip
For any $u,v\in\LPSH\Cn$ we have, obviously, the relations
\beq\label{eq:order}
u\le v+ O(1)\Rightarrow \LCT(u)\ge \LCT(v),\  \LCTh(u)\ge \LCTh(v), 
\eeq
\[
\LCT(tu)=t^{-1}\LCT(u) \quad \forall t>0,
\]
and $\LCT(C)=\LCTh(C)=+\infty$ for any constant $C$.

For functions of the class $\LPSHs\Cn$, we have, in addition, 

\begin{proposition}\label{prop:comparison} 
\begin{enumerate}
\item[(i)]  $\LCT(u+v)\le \LCT(\max\{u,v\})\le \min\{\LCT(u),\LCT(v)\}$ for any $u,v\in\LPSHs\Cn$;
\item[(ii)] if $u\in \LPSH\Cn$ and $v\in \LPSHs\Cn$ are such that 
\[ \limsup_{z\to\infty}\frac{u(z)}{v(z)}\le \sigma<\infty,
    \]
    then $\LCT(u)\ge \sigma^{-1}\LCT(v)$. 
\end{enumerate}
\end{proposition}

As an immediate consequence of Proposition~\ref{prop:comparison}(ii) and the elementary computation 
    $\LCT(\log|z|)=n$, we get
\begin{corollary}\label{cor:logtype}
    For any $u\in\LPSHs\Cn$, we have 
    \[
    0<\frac{n}{\sigma_u}\le \LCT(u)\le \frac{n}{\LL_\infty(u)}\le +\infty,
    \]
        where $\sigma_u$ is the logarithmic type \eqref{eq:logtype} of $u$ and  \( \LL_\infty(u)\ge 0 \)
        is the  Łojasiewicz exponent \eqref{eq:Le} of $u$ at $\infty$.
        In particular,
        \[
    \frac{n}{\deg P}\le \LCT(P)\le \frac{n}{\LL_\infty(P)},
    \]
    for any map $P\in\cPs$.
\end{corollary}

\begin{remark}
    The lower bounds
    \beq\label{eq:lboundtype}
    \LCT(u)\ge\LCTh(u)\ge \frac{n}{\sigma_u}
    \eeq
    are true for any $u\in\LPSH\Cn$. More precisely, $e^{-u^+}\not\in L^2(\Cn)$ if $\sigma_u\le n$.
\end{remark}

\begin{remark}
    By Skoda's inequalities \eqref{eq:Skoda}, the integrability index $\lambda_\zeta(u)$ at any point $\zeta\in\Cn$ is comparable to the Lelong number $\nu_\zeta(u)$. At infinity, the role of Lelong number is played by the logarithmic type $\sigma_u$. By Corollary~\ref{cor:logtype}, we always have the upper bound $\lambda_\infty(u)\le \sigma_u/n$, however there is no uniform  bound $\lambda_\infty(u)\ge C \sigma_u$ even in the class $\LPSHs\Cn$. Indeed, for $v_N(z)=\log\max\{|z_1|,|z_2^N|,|z_1z_2^N|,1\}$, we have $\lambda_\infty(v_N)=1$ (which follows, for example, from results of Section 6) and $\sigma_{v_N}=N+1$.
\end{remark}

\begin{remark}
    We do not know if $\LCT(u)<\infty$ for any $u\in\LPSHs\Cn$, while this is definitely so if $u$ has positive  Łojasiewicz exponent at $\infty$.
\end{remark}

We can also compare $\LCT(u)$ with that of its {logarithmic indicator} $\Psi_u$ \eqref{eq:ind}.  In view of \eqref{eq:ind_bound} and \eqref{eq:order}, we have

\begin{corollary}\label{cor:comp_ind}
        $\LCT(u)\ge \LCT(\Psi_u)$ for any $u\in\LPSHs\Cn$.
\end{corollary}
To make this useful, we need to compute $ \LCT(\Psi_u)$, which we will do in the next section.

\section{LCT of indicators}
Logarithmic indicators $\Psi_u$ of $u\in\LPSH\Cn$ (see Section~\ref{ssec:ind}) have nice properties that allow an easy computation of their characteristics like the Lelong numbers, Monge-Ampère measures, log canonical thresholds at $0$, etc., see \cite{Rash01, Rash03, Rash17}. Here we will show how to compute their log canonical thresholds at $\infty$ in terms of the indicator diagrams $\Gamma_{u,\infty}$. 

\medskip
Let $\Psi\in\LPSHs\Cn$ be an abstract {\it indicator}, that is, a psh function satisfying (\ref{eq:ind_hom}), and let $\psi$ be the corresponding convex function on $\Rn$. For such a  $\Psi$, we have $\{z\in\Cn:\: \Psi(z)<0\}$ is either $\D^n$ (the unit polydisk) or  empty, while $\{z\in\Cn:\: \Psi(z)>0\}=\Cn\setminus {\overline\D^n}$. Equivalently,  $\{t\in\Rn:\: \psi(t)<0\}=\Rn_-$ or empty, and $\{t\in\Rn:\: \psi(t)>0\}=\Rn\setminus {\overline\Rnm}$. Indeed, if $\psi(t^0)\le 0$ for some $t^0\not \in {\overline\Rnm}$, then $\psi(ct^0)\not\to +\infty$ as $c\to +\infty$.

Then $\Psi=\Psi^+:=\max\{\Psi,0\}$ on $\Cn\setminus\D^n$, $\LCT(\Psi)=\LCT(\Psi^+)$, and the indicator diagram $\Gamma^+\subset \overline\Rnp$ of $\Psi^+$ contains the set $\overline{U\cap\Rnp}$ for some neighborhood $U$ of $0$.

Denote by ${\bf 1}$ the single-point set $\{(1,\ldots,1)\}\in\Rnp$. Take any $c>0$. By \eqref{eq:indsupp}, we have $c\psi^+=h_{c\Gamma^+}$ and $h_{-{\bf 1}}(t)=-\sum_k t_k$ for the set $-{\bf 1}=\{(-1,\ldots,-1)\}\in\Rnm$.
Then
\begin{eqnarray*}
  & & \frac1{(2\pi)^n}\int_{\Cn}e^{-2c\Psi^+} \,\beta_{n}
=
\int_{\Rn}e^{-2[-\sum_k t_k+c\psi^+(t)]}
\,dt \\
&=&\int_{\Rn}e^{-2[h_{-{\bf 1}}(t)+h_{c\Gamma^+}(t)]}
\,dt =
\int_{\Rn}e^{-2h_{{-{\bf 1}}+c\Gamma^+}(t)}
\,dt.
\end{eqnarray*}
By Lemma~\ref{lem:support} below, the last integral converges if and only if $h_{{\bf 1}+c\Gamma^+}>0$ on the unit sphere $S$, which is equivalent to 
\beq\label{eq:cpsi} c\psi^+(t)> \sum_k t_k\quad \forall t\in S.
\eeq
Denote 
\beq\label{eq:H+}
H_+=\{t\in\Rn:\: \sum_k t_k> 0\}.
\eeq
Since (\ref{eq:cpsi}) trivially holds on $S\setminus H_+$ and the both sides of (\ref{eq:cpsi}) are positive homogeneous, the convergence of the integral is equivalent to 
\beq\label{eq:c on H+} c\psi^+(t)> \langle t,\bone\rangle= \sum_k t_k\quad \forall t\in H_+,
\eeq
or, which is the same,
\[
c\psi^+(t)> 1\quad \forall t\in H_1, 
\]
where
\beq\label{eq:H1}
H_1=\{t\in\Rn:\: \sum_k t_k=1\}.
\eeq

In addition, condition \eqref{eq:c on H+} means that there exists $\alpha<1$ such that $c\alpha\,\psi^+(t)\ge \sum_k t_k$ for all $t\in H_+$ and thus for all $t\in\Rn$. Since $c\alpha\,\psi$ is the supporting function of the set $\Gamma_{c\alpha\Psi^+}$ and $\sum_k t_k$ is the supporting function of the single-point set ${\bf 1}$, this means precisely ${\bf 1}\subseteq \Gamma_{c\alpha\Psi^+}$ and hence, $c^{-1}{\bf 1}$ belongs to the interior of $\Gamma_{\Psi^+}$.

This proves

\begin{proposition}\label{prop:ind_Kiselman} 
Let $\Psi\in\LPSHs\Cn$. Then 
\beq\label{eq:conv}
\int_{\Cn}e^{-2c\Psi^+} \,\beta_{n}<\infty
\eeq
if and only if 
\beq\label{eq:ind_Kiselman1}
c>\sup\left\{ \frac{\sum_kt_k}{\psi(t)}:\: t\in H_+\right\},
\eeq
so
\beq\label{eq:ind_Kiselman2}
\LCT(\Psi)=\sup\left\{ \frac{\sum_kt_k}{\psi(t)}:\: t\in H_+\right\}= \frac{1}{\inf\{\psi(t):\: t\in H_1\}}.
\eeq
In addition, inequality (\ref{eq:ind_Kiselman1}) means that the point $c^{-1}\bone$ lies in the interior of the indicator diagram $\Gamma_{\Psi^+}$ of $\Psi^+$. 
\end{proposition}

In the proof, we have used the following, probably known, fact.

\begin{lemma}\label{lem:support} 
Let $h_K$ be the supporting function of a bounded convex body $K\subset\Rn$. Then
\beq\label{eq:finite int}
\int_{\Rn} e^{-h_K(t)}\,dt<\infty 
\eeq
if and only if $h_K(t)>0$ for all  $t\neq 0$.
\end{lemma}

\begin{proof}
    Evidently, (\ref{eq:finite int}) implies $h_K(t)\ge 0$ in $\Cn$. If $h_K(t)>0$ for all  $t\neq 0$, then $h_K\ge\delta>0$ on the unit sphere $S$, and hence
    \begin{eqnarray*}
    \int_{\Rn} e^{-h_K(t)}\,dt &=& \int_0^\infty\int_{S} e^{-r\,h_K(s)}r^{n-1}\,dS\,dr \\ 
    &\le& \int_0^\infty\int_{S} e^{-\delta r}r^{n-1}\,dS\,dr<\infty.
    \end{eqnarray*}
To prove the converse, let $h_K(t_0)=0$ for some $t_0\neq 0$. Then, by a change of coordinates, we can assume $t_0=(1,0,\ldots,0)$. Let $\gamma=\sup h_K$ on the set $E=\{1\}\times \B_{n-1}\subset\Rn$. If $\gamma=0$, then $h_K\equiv 0$ on the whole cone $C_E=\cup_{r>0}\, rE$ and integral in (\ref{eq:finite int}) diverges, so we assume $\gamma>0$. Since the function $h_K(1,s)$ is convex on $E$, it is dominated there by $v(s)=\gamma |s|$ and so, $h_K(t)=h_K(t_1,t')=t_1h(1,t'/t_1)\le \gamma|t'|$ on $C_E$. Therefore,
\[
\int_{\Rn} e^{-h_K(t)}\,dt\ge \int_{C_E} e^{-\gamma |t'|}\,dt_1\,dt'=\infty,
\]
and the conclusion follows.  
\end{proof}

\begin{remark}
    Proposition~\ref{prop:ind_Kiselman} need not be true for indicators not from the class $\LPSHs\Cn$. Indeed, for $\Psi=\log|z_1\ldots z_n|$, the right hand side of \eqref{eq:ind_Kiselman1} equals $1$ while $\LCT(\Psi)=\LCT(\Psi^+)=+\infty$.
\end{remark}

\medskip

\begin{example}    Given $a=(a_1,\ldots,a_n)\in\Rnp$, consider the function      \beq\label{eq:phia}
    S_a(z)=\max_k\frac{\log|z_k|}{a_k}.
    \eeq
    Evidently, $S_a=\Psi_{S_a}\in\LPSHs\Cn$, and its convex image
    $$s_a(t)=
    \max_k\frac{t_k}{a_k}
    $$ 
is the supporting function of the $(n-1)$-dimensional simplex
 \[
    \Sigma_a=\left\{b\in{\overline\Rnp}:\: \sum_k \frac{b_k}{a_k}=1\right\},
    \]
while $s_a^+(t)$
    is the supporting function of the $n$-dimensional simplex
 \[
    \Delta_a=\left\{b\in{\overline\Rnp}:\: \sum_k \frac{b_k}{a_k}\le 1\right\}.
 \]   
 Since $\inf\{s_a(t):\: t\in H_1\}=(\sum_k a_k)^{-1}$, we have, by Proposition~\ref{prop:ind_Kiselman},
    \beq\label{eq:lct phia}
    \LCT(S_a)=\sum_k a_k;
    \eeq
    note that it equals $c_0(S_a)$.
    We have also
    \beq\label{eq:Sa}
    \m_n(S_a):=(dd^c S_a)^n(\Cn)=(a_1\cdot\ldots\cdot a_n)^{-1}
    \eeq
and, for any $u\in\LPSH\Cn$,
    \beq\label{eq:Psi_phia}
\psi_u(a)=\limsup_{z\to\infty}\frac{u(z)}{S_a(z)},\quad a\in\Rnp,
    \eeq
    see \cite{Rash01}.
    \end{example}

\medskip

It is not in general true that $\inf\{\psi^+(t):\: t\in H_1\}$ in Proposition~\ref{prop:ind_Kiselman}  is attained on $H_1\cap\Rnp=\Sigma_\bone$. 

\begin{example} Let $\psi(t)$ on $\R^2$ be defined as 
$\psi(t)=\max\{ 4t_1, 4t_1+2t_2,t_2,0\}$.
Then $\psi(t)=4t_1+2t_2$ on $\R_+^2$, so 
$\inf\{\psi(t):\: t\in \Sigma_\bone\}=2$. At the same time, $\psi(t)= t_2$ on the half-line $\{t_2=-4t_1, \ t_1<0\}$, so $\psi(-1/3,4/3)=4/3<2$. 
\end{example}

On the other hand, for some classes of indicators, the two infimums are guaranteed to coincide. We recall the following notion (e.\,g., \cite{BoLe}, \cite{MSSS}).

\begin{definition}\label{def:lower}
    A compact convex set in $\Gamma\subset\overline\Rnp$
is a {\it lower set} if $$a\in\Gamma \implies b\in\Gamma$$ for any $b\in\Rnp$ such that $b_k<a_k$, $1\le k\le n$.
\end{definition}

\begin{lemma}\label{lem:lowerset} Let $\Psi(z)\in\LPSHs\Cn$ be an indicator such that $\psi^+(t)$ is the supporting function of a lower set (in other words, if its indicator diagram is a lower set). Then 
\beq\label{eq:inf}
\inf\{\psi^+(t):\: t\in H_1\}=\inf\{\psi^+(t):\: t\in \Sigma_\bone\}.
    \eeq
\end{lemma}

\begin{proof}
    Equality \eqref{eq:inf} is equivalent to saying that for any 
    $$c>\inf\{\psi^+(t):\: t\in \Sigma_\bone\},$$ the point $c^{-1}\bf 1$ belongs to the interior of the indicator diagram $\Gamma^+$. 
    
    Consider the translation $\Gamma'=\Gamma^+-c^{-1}\bf 1$ of $\Gamma^+$.  Then 
    \beq\label{eq:c1}
    h_{\Gamma'}(t)=\psi(t)-c^{-1}\sum_k t_k 
     > 0
    \quad \forall t\in\Rnp.
    \eeq  
As $\Gamma^+$ is a lower set, then either $\Gamma'\cap\Rnp$ contains a neighborhood of the origin, or it is empty. In the former case, $c^{-1}\bf 1$ belongs is an interior point of $\Gamma^+$. In the latter case, there exists a hyperplane $l$ separating  $\Rnp$ and $\Gamma'$. Since $l\cap\Rnp=\emptyset$, its upward normal $t$ belongs to $\overline\Rnp$, and the separation property contradicts \eqref{eq:c1}. 
\end{proof}

\begin{remark}
    Actually, a much stronger result is proved in \cite[Thm. 5.8]{MSSS}: for a convex function $\psi$ of a linear growth on $\Rn$, $\psi^+(t)=\psi^+(t^+)$ for all $t\in\Rn$ if and only if the indicator diagram of $\psi^+$ at infinity is a lower set; some other characterizations of lower sets were given there as well. As shown in \cite{Sno}, Lipschitz continuity of $\Psi(z)$ is also one of such  characterizations. 
\end{remark}

\begin{corollary}\label{cor:lctlowset}
    If an indicator $\Psi\in\LPSHs\Cn$ is the logarithmic supporting function of a lower set, then  
    \[
    \LCT(\Psi)=\sup\left\{ \frac{\sum_ka_k}{\psi(a)}:\: a\in \Rnp\right\}= \frac{1}{\inf\{\psi(a):\: a\in \Sigma_\bone\}}.
    \]
\end{corollary}

\section{Toric functions and monomial maps}

Our next step is to extend (\ref{eq:ind_Kiselman1})-(\ref{eq:ind_Kiselman2}) to {\it toric (multicircled)} psh functions $u$ in $\LPSHs\Cn$: $u(z)=u(|z_1|,\ldots,|z_n|)$. We denote the class of all such functions by $\TLPSHs\Cn$; we will also use the denotations $\TPSH\Cn$ and $\TLPSH\Cn$ for the corresponding toric subcones of $\PSH\Cn$ and $\LPSH\Cn$. Important examples of toric functions are $u=\log|P|$ for {\it monomial maps} $P:\Cn\to\C^q$ (maps whose components are monomials). 

Any $u\in\TPSH\Cn$ generates a convex function 
\[
\widehat u(t):=u(e^{t_1},\ldots, e^{t_n})
\]
on $\Rn$ (the {\it convex image} of $u$), increasing in each $t_k$; conversely, any such function on $\Rn$ produces the corresponding function from $\TPSH\Cn$. In Section~\ref{ssec:ind}, we used it for the indicators $\Psi_u$; in accordance with the present denotation, $\psi_u=\widehat \Psi_u$.

The condition $u\in\TLPSH\Cn$ means that $\widehat u$ is of linear growth at $+\infty$:  
$\widehat u(t)\le C\max_kt_k+O(1)$  as $\max_kt_k\to+\infty$
with some $C\ge 0$, and the condition 
$u\in\TLPSHs\Cn$ means that, in addition, 
\[
\{t:\: \widehat u(t)<C\}\subset \{t:\: \max_k t_k\le T\}
\]
for some $C, T\in\R$.

The indicator of any $u\in\TLPSH\Cn$ can be constructed as follows. First, we switch from $u(z)$ to $\widehat u(t)$. By the convexity of $\widehat u(t)$ in $t$, the ratio
\beq\label{eq:ratio}
v_c(t):=\frac{\widehat u(ct)-\widehat u(t)}{c-1}, \quad c>0,
\eeq
is increasing in $c$ for each fixed $t\in\Rn$, and its limit, as $c\to +\infty$,  is $\psi_u=\widehat\Psi_u$. 
Moreover, given {\sl any} $u\in\LPSH\Cn$, one constructs its indicator $\Psi_u$ as the indicator of the toric function 
\[
\int_{\Tn}u(z\zeta)\,d\sigma(\zeta),
\]
where $d\sigma$ is the Haar measure on the distinguished boundary $\Tn$ of the unit polydisk and $z\zeta=(z_1\zeta_1,\ldots, z_n\zeta_n)$, see details in \cite{Rash01}.

A lot of results on asymptotics of functions from $\TLPSHs\Cn$ can be deduced from the corresponding statements about their indicators due to the following

\begin{theorem}\label{thm:lbound}
    Let $u\in\TLPSHs\Cn$. Then for any $\epsilon>0$ there exists $R>1$ such that
    \beq\label{eq:lbound}
u(z)\ge (1-\epsilon)\Psi_u(z)+C,\quad z\in\Cn\setminus \D_R^n,
\eeq
with $C=\sup\{u(z):\: z\in \D_1^n\}$.
\end{theorem}

\begin{proof}
    It suffices to prove this assuming $u\ge 0$ outside $\D^n$ and so, $\widehat u\ge 0$ on $\Rn\setminus\Rnm$. Since the function $v_c(t)$ \eqref{eq:ratio} is continuous in $t$ and increasing in $c$ to a continuous function $\widehat\Psi_u$, Dini's lemma implies the uniform convergence on the compact set $\Lambda=\{t\in \Rn\setminus\Rnm:\:  \sum_k |t_k|=1\}$. Therefore, for any $\epsilon>0$, we have, for all $c>c_0>1$ and all $t\in \Lambda$,
\[
\widehat u(ct)\ge \widehat u(t)+(1-\epsilon/2)\frac{c-1}{c}\psi_u(ct) \ge (1-\epsilon/2)\frac{c-1}{c} \psi_u(ct).
\]
Then there exists $R>1$ such that 
\[
\widehat u(t)\ge (1-\epsilon)\psi_u(t)
\]
for all $t\in \Rn\setminus\Rnm$ with $\max_k t_k\ge \log R$. Passing to the toric psh functions $u$ and $\Psi_u$, we get
\[
u(z)\ge (1-\epsilon)\Psi_u(z),\quad z\in\Cn\setminus \D_R^n.
\]
which completes the proof.
\end{proof}

Since $u\le \Psi_u + C$ for any $u\in\LPSH\Cn$, 
Theorem~\ref{thm:lbound} and Proposition~\ref{prop:comparison} give us immediately the equality
\beq\label{eq:torindeq}
\LCT(u)=\LCT(\Psi_u) \quad \forall u\in\TLPSHs\Cn,
\eeq

 Furthermore, together with Proposition~\ref{prop:ind_Kiselman} it implies

\begin{theorem}\label{thm:toric_Kiselman}
   For any $u\in\TLPSHs\Cn$ and a compact $K\supset\D^n$ , we have
   \beq\label{eq:conv_toric}
\int_{\simk}e^{-2cu} \,\beta_{n}<\infty
\eeq
if and only if 
\beq\label{eq:toric_Kiselman1}
c>\sup\left\{ \frac{{\langle t,\bone\rangle}}{\psi_u(t)}:\: t\in H_+\right\}
\eeq
(in other words, $c^{-1}\bone\in {\rm int}\,\Gamma_{u,\infty}$), 
so
\begin{eqnarray*}\label{eq:toric_Kiselman2}
\LCT(u) &=& \LCT(\Psi_u)=\sup\left\{ \frac{{\langle t,\bone\rangle}}{\psi_u(t)}:\: t\in H_+\right\} \\
&=& \inf\{c>0:\: c^{-1}\bone\in \Gamma_{u,\infty}\};
\end{eqnarray*}
here the set $H_+$ is defined by (\ref{eq:H+}), $\psi_u(t)$ is the convex image \eqref{eq:psiu} of the indicator $\Psi_u(z)$, and $\Gamma_{u,\infty}$ is the indicator diagram of $u$ at infinity \eqref{eq:indicdiaginfty}.
\end{theorem}

\begin{remark}\label{rem:locglob}
    In terms of the integrability index, the formulas rewrite as
    \beq\label{eq:KisHowglob}
    \lambda_\infty(u)=\inf\{ \psi_u(t):\: t\in H_1\}
    = \sup\{\lambda>0:\: \lambda\bone\in \Gamma_{u,\infty}\},
    \eeq
     where $H_1$ is defined by (\ref{eq:H1}). Their original local versions are due to Kiselman \cite{Kiselman} and Howald \cite{Howald}, see also \cite{Gue} and \cite{Rash13}:
    \beq\label{eq:KisHowloc}
\lambda_0(u)=\sup\left\{ \frac{\psi_u(t)}{\langle t,\bone\rangle}:\: t\in \Rnm\right\}
    = \inf\{\lambda>0:\: \lambda\bone\in \Gamma_{u,0}\},
    \eeq
   where $\Gamma_{u,0}$ is the indicator diagram \eqref{eq:G0} of $u$ at $\zeta=0$.
\end{remark}

By comparing \eqref{eq:KisHowglob} and \eqref{eq:KisHowloc}, one deduces the following relation between the local and global thresholds of any toric function from  $\LPSHs\Cn$.

\begin{corollary}\label{cor:locglobtor} Let $u\in\TLPSHs\Cn$, then $ \LCT(u)\le c_0(u)$, with equality if and only if the indicator $\Psi_u(z)$ equals $S_a(z)$ \eqref{eq:phia} for some $a\in\Rnp$ (equivalently, $\Gamma_{u,\infty}$ is a simplex). 
    \end{corollary}
    
\begin{proof}
    
    Consider the $(n-1)$-dimensional simplex
    \[
    \Sigma_a=\left\{b\in{\overline \Rnp}:\: \sum_k \frac{b_k}{a_k}=1\right\}
    \]
    with $a_k=-\lim \psi_u(t)>0$ as $t_k\to-1$ and $t_j\to-\infty$ for all $j\neq k$. 
    In view of \eqref{eq:G0} and monotonicity of $\psi_u$, any its vertex $a^{(k)}$ satisfies $\langle a^{(k)},t\rangle \le \psi(t)$ for all $t\in\Rnm$ and, therefore, belongs to the convex set $\Gamma_{u,0}$, which gives us $\Sigma_a\subset\Gamma_{u,0}$. In order to check $\Sigma_a\subset\Gamma_{u,\infty}$, it suffices to show that each $a^{(k)}$ satisfies 
    $\langle a^{(k)},t\rangle \le \psi(t)$ for all $t\in \Rn\setminus\Rnm$. When $t_k<0$, this follows as in the above case $t\in\Rnm$. For $t_k>0$, we use, in addition, convexity of $\psi_u$ in $t_k$ and the property $\psi_u(t)=0$ on $\partial\Rnm$, which gives us  $a_k\le \lim \psi_u(t)$ as $t_k\to 1$ and $t_j\to-\infty$ for all $j\neq k$.

    The ray $\{\lambda\bone:\: \lambda>0\}$ intersects the set $\Sigma_a$ at the point $\lambda_a\bone$ with $\lambda_a^{-1}=\sum_k a_k^{-1}$. Therefore, by \eqref{eq:KisHowglob} and \eqref{eq:KisHowloc},
    \(  \lambda_0(u)\le \lambda_a \le  \lambda_\infty(u)\),
    which completes the proof.
\end{proof}

\begin{remark}
    If $u\in\LPSHs\Cn$ is such that $e^{-cu}\in L^2(\Cn)$ for some $c>0$, then  $\LCT(u)\le c<c_\zeta(u)$ for every $\zeta\in\Cn$. We believe that  the relation $\LCT(u)\le c_\zeta(u)$ holds true for every $\zeta\in\Cn$ and any $u\in\LPSHs\Cn$ (for $n=1$, this is so by Corollary~\ref{cor:locglob1}). 

    Note also that the relation  $\LCT(u)\le c_0(u)$  need not be true if $u\not\in\LPSHs\Cn$, even if $u$ is toric. For example, $c_0(\log|z_1|)=1$ while $\LCT(\log|z_1|)=+\infty$.
\end{remark}

As a simple consequence of theorem~\ref{thm:toric_Kiselman}, we get the openness property for toric psh functions:

\begin{corollary}
    For any $u\in \TLPSHs\Cn$, $e^{-\LCT(u)u}\not\in L^2(\infty)$.
\end{corollary}

\begin{remark}
Openness of $\{c>0:\: e^{-cu}\in L^2(\zeta)\}$ for any function $u$ plurisubharmonic near a point $\zeta\in\Cn$ was the famous Openness Conjecture \cite{DK} proved in \cite{BB}.
\end{remark}
 
\begin{remark}
We can also describe the set of $c>0$ such that 
\beq\label{eq:globalint}
e^{-cu}\in L^2(\Cn)
\eeq
for $u\in \TLPSHs\Cn$. Since the only possible singularity it can have is at $\zeta=0$, \eqref{eq:globalint} takes place if and only if $\LCT(u)<c<c_0(u)$. 
\end{remark}

\medskip

In the polynomial case $u=\log|P|$, $P=(P_1,\ldots,P_q)\in\cPs$, the indicator diagram $\Gamma_{\log|P|,\infty}$ is the Newton polytope $\Gamma_{P,\infty}$ (see the end of Section~\ref{ssec:ind}), and the indicator $\Psi_{\log|P|}=\max_k \Psi_{\log|P_k|}$. This makes computation of $\LCT(P)$ using Theorem~\ref{thm:toric_Kiselman} simple and explicit. 
Since $\log|P|\in\TLPSHs\Cn$ for every {\sl monomial map} $P\in \cPs$, this gives us an effective way of evaluating $\LCT(P)$ for such maps.
Moreover, this can be extended to a larger class of polynomial maps.

\begin{definition}\label{def:Sp} {\rm \cite{BH17}}
Given a polynomial map $P=(P_1,\dots, P_q):\C^n\to \C^q$ and $h\in\C[z_1,\dots, z_n]$, we say that $h$ is
{\it special with respect to $P$} when
\begin{equation}
\vert h(z)\vert \leq C \Vert P(z)\Vert
\end{equation}
for all $\Vert z\Vert \gg 1$ and some constant $C>0$.

We denote by $\Sp(P)$, or by $\Sp(P_1,\dots, P_q)$, the set of all polynomials $h\in\C[z_1,\dots, z_n]$
such that $h$ is special with respect to $P$, and we will refer to $\Sp(P)$ as the {\it special closure of $P$}.
\end{definition}

We remark that $\Sp(P)$ is a $\C$-vector subspace of $\C[z_1,\dots, z_n]$ containing $\Lin(P)$, the $\C$-vector subspace of $\C[z_1,\dots, z_n]$
generated by the component functions of $P$. Analogously to the usual notion of an integrally closed ideal, we say that $P$ is {\it specially closed} if $\Lin(P)=\Sp(P)$.

We also recall two notions going back to \cite{Ku}.

\begin{definition}
    (a) A polynomial map $P:\Cn\to\C^q$ is convenient (commode) if for any $k\le n$ there exists a positive integer $r_k$ such that $r_ke_k\in \Gamma_{P,\infty}$, where $\{e_1,\ldots,e_n\}$ is the canonical basis in $\Rn$. 
    
     (b) We say that $P$ is Newton non-degenerate at infinity if, for any face $\Delta$ of $\Gamma_{P,\infty}$, not passing through $0$, the map $P^\Delta=(P_1^\Delta,\ldots,P_q^\Delta)$ has no zeros outside the union of coordinate planes; here 
    \[
    P_k^\Delta = \sum_{J\in \supp(P_k)\cap\Delta} c_Jz^J.
    \]
    is the $\Delta$-reduction of the polynomial $ P_k = \sum_{J} c_Jz^J.$
    \end{definition}

Let $P$ be a polynomial map, convenient and Newton
non-degenerate at infinity. Then automatically $P\in \mathcal P^*$, since the \L ojasiewicz exponent of $P$ at infinity is positive in this
case \cite{BH17}.  Under these conditions, the special closure of $P$ is equal to the $\C$-vector space generated by those monomials $z^J$ such that $J\in \Gamma_{P,\infty}$. In particular, we have that
$$
\LL_\infty(P)=\min\{r_1,\dots, r_n\}
$$
where $r_1,\dots, r_n$ are the positive integers such that $r_1e_1,\dots, r_ne_n$ belong to the Newton boundary of $P$ (the union of the faces of  $\Gamma_{P,\infty}$ not passing through the origin).

In general, we have $c_\infty(P)=c_\infty(Q)$, where $Q:\C^n\to \C^s$ denotes any polynomial map whose components
generate, as a vector space, the special closure of $P$. If we assume that $P$ is Newton non-degenerate, then $Q$ can be taken as
any monomial map whose set of component functions is equal to $\{z^k: k\in \Gamma_{P,\infty}\}$ \cite[Corollary 4.10]{BH17}.
Therefore, we conclude that 
\beq\label{eq:spcl}
\LCT(P)=\LCT(Q)=\LCT(\Psi_{\log|Q|})=\LCT(\Psi_{\log|P|}),
\eeq
which proves

\begin{corollary}\label{cor:NNDlct}
Let $P:\C^n\to \C^q$ be a polynomial map, convenient and Newton
non-degenerate at infinity. Then
$$
c_\infty(P)=\frac{1}{\max\{t>0: t\bone\in\Gamma_{P,\infty}\}}
$$
(and thus is a rational number).
\end{corollary}

\begin{remark}
Let $P:\Cn\to\C^q$ be a polynomial map and let us fix coordinates $z_1,\dots, z_n$ in $\C^n$.
Then we can consider the set
$$
A_P=\{J\in\Z^n_{\geq 0}: z^J\in\Sp(P)\}.
$$
In some cases, this set can be estimated or even explicitly described. Let $G:\C^n\to \C^q$ denote a polynomial map whose set of components is precisely $\{z^J: J\in A_P\}$. We automatically have
$$
c_\infty(G)\geq c_\infty(P)\geq c_\infty(P^0)
$$
where $P^0$ denotes any polynomial map whose set of component functions is $\supp(P_1)\cup\cdots \cup \supp(P_q)$.
We can show several examples where the bounds $c_\infty(G)$ and $c_\infty(P^0)$ coincide, so we can deduce the exact value of
$c_\infty(P)$ in some cases where $P$ is not Newton nondegenerate at infinity. 
\end{remark}

\section{Multipliers}\label{sec:multipliers}

It is easy to see that any entire function $f$ in $\Cn$ such that $|f|e^{-u}\in L^2(\infty)$ for some $u\in\LPSH\Cn$ is necessary a polynomial: since $u^+(z)\le\sigma_u\log^+|z|+C$, we have, for any $r>0$ and $z\in\Cn\setminus\B_1$,
\begin{eqnarray}\label{eqn:polynome}
 |f(z)|^2 &\le& \frac{1}{\tau_nr^{2n}}    
 \int_{\B_r(z)}|f|^2\,\beta_n \le \frac{e^{\sup\{2u^+(\zeta):\: \zeta\in\B_r(z)\}}}{\tau_nr^{2n}}
 \int_{\B_r(z)}|f|^2 e^{-2u^+}\,\beta_n \nonumber \\
 &\le& C_1 (|z|+r)^{2\sigma_u} r^{-2n}. 
\end{eqnarray}

\begin{definition}
Given $u\in\LPSH\Cn$, denote by $\cP(u)$ the collection of all polynomials ({\it multipliers})  $P:\Cn\to\C$ such that
$|P|e^{-u}\in L^2(\infty)$.
\end{definition}

Letting $r\to\infty$ in \eqref{eqn:polynome} shows that $\cP(u)=\{0\}$ if $\sigma_u<n$; this is also true if $\sigma_u=n$ because, in this case, $e^{-u}\not\in L^2(\infty)$ (exactly as in the one-dimensional case, see Theorem~\ref{thm:C}).
For $\sigma_u>n$, we have 
 
\begin{proposition}
    Let $\sigma_u>n$. Then any its multiplier is a polynomial of degree at most $\sigma_u-n$. Furthermore, if the Łojasiewicz exponent ${\LL}_\infty(u)> n$ and the local log canonical thresholds $c_\zeta(u)>1$ for all $\zeta\in\Cn$, then any polynomial of degree smaller than $\LL_\infty(u)-n$ belongs to $\cP(u)$.
\end{proposition}

\begin{proof}
Since, for any $t>0$ and $b>c>0$, 
\[
\inf_{r>0} C_1(t+r)^{b} r^{-c} = C_2 t^{b-c},
\]
bound \eqref{eqn:polynome} implies the first assertion. 

If $c_\zeta(u)>1$ for any $\zeta\in\Cn$, then 
\[
\int_{|z|\le R} |P|^2 e^{-2u}\,\beta_n < \infty 
\]
for any $R>0$ and any polynomial $P$, while for $P$ of degree $k<\LL_\infty(u)-n$ the bound $u(z)\ge (\LL_\infty(u)-\epsilon)\log^+|z|-C_\epsilon$ implies
\[
\int_{|z|>R} |P|^2 e^{-2u}\,\beta_n \le C \int_R^\infty 
r^{2(k-\LL_\infty(u)+\epsilon+n)-1}\,dr,
\] 
which gives us the second assertion.
\end{proof}

Any nontrivial $\cP(u)$ is a finite-dimensional $\C$-vector space with the property {\sl $P\in\cP(u)$, provided $|P(z)|\le C|Q(z)|$ for some $Q\in\cP(u)$, $C>0$ and all $z$ with $|z|\gg 1$}. In other words, any space $\cP(u)$ is {\it specially closed}: $\Sp(\cP(u))=\cP(u)$ (see Definition~\ref{def:Sp}).

\medskip

We will call a nontrivial vector space $V$ of polynomials $P:\Cn\to \C$ {\it monomial} if it can be generated by monomials: there exist finitely many monomials $z^J\in V$ such that any $P\in V$ can be represented as their linear combination: $P=\sum_J c_Jz^J$, $c_J\in\C$.

\begin{proposition}\label{prop:monmult}
    If $u\in\TLPSHs\Cn$  and $\LCT(u)<1$, then $\cP(u)$ is monomial.
\end{proposition}

\begin{proof}
    Similarly to the proof of \cite[Lemma 3.7]{Rash13}, let $P=\sum c_J z^J\in\cP(u)$. Since monomials are orthogonal on polydisks centered at the origin, we have
\begin{eqnarray*}
     \int_{\Cn\setminus\D_R^n} |P|^2e^{-2u}\,dV &= & \sum_{J,I}c_J\overline{c_I}\int_{\Cn\setminus\D_R^n} z^J
\bar z^I e^{-2u(|z_1|,\ldots,|z_n|)}\,dV(z) \\
&=& \sum_\alpha |c_J|^2\int_{\Cn\setminus\D_R^n} |z^J|^2e^{-2u}\,dV,
\end{eqnarray*}
which implies $z^\alpha\in\cJ(u)$ for any $\alpha$ such that $c_\alpha\neq 0$.
\end{proof}

\medskip
Applying Proposition \ref{prop:monmult} to any indicator $\Psi\in\LPSHs\Cn$ with $\LCT(\Psi)<1$, we find that $\cP(\Psi)$ is monomial. Its description is given by

\begin{proposition}\label{prop:multind}
    The multiplier space $\cP(\Psi)$ of an indicator $\Psi\in\LPSHs\Cn$ with $\LCT(\Psi)<1$ is generated by monomials $z^J$ such that $J+ {\bf 1}$ belong to the interior of $\Gamma_{\Psi^+}$.
\end{proposition}

\begin{proof}
    We can assume $\Psi=\Psi^+$, so the corresponding convex function $\psi$ is the supporting function of the set $\Gamma=\Gamma_{\Psi^+}$. 
    
    Take any $J\in\Z_+^n$. By repeating arguments of the proof of Proposition~\ref{prop:ind_Kiselman}, we have
    \begin{eqnarray*}
  & & \frac1{(2\pi)^n}\int_{\Cn}|z^J|^2 e^{-2\Psi} \,\beta_{n}
=
\int_{\Rn}e^{-2[-\langle J+{\bf 1},t \rangle+\psi(t)]}
\,dt \\
&=&\int_{\Rn}e^{-2[h_A(t)+h_{\Gamma}(t)]}
\,dt =
\int_{\Rn}e^{-2h_{A+\Gamma}(t)}
\,dt
\end{eqnarray*}
where $h_{A}(t)$ is the supporting function of the single-point set $A=\{-J-{\bf 1}\}$.
By Lemma~\ref{lem:support}, the last integral converges if and only if $h_{A+\Gamma}>0$ on the unit sphere $S$, which is equivalent to 
\[
\psi(t)> \sum_k t_k\quad \forall t\in S,
\]
 which, precisely as in the proof of Proposition~\ref{prop:ind_Kiselman}, means $J+{\bf 1}\in {\rm Int}(\Gamma)$.
\end{proof}

\begin{theorem}\label{theo:How2}
    If $u\in\TLPSHs\Cn$ and $\LCT(u)<1$, then $\cP(u)=\cP(\Psi_u)$  and it is generated by monomials $z^J$ such that $J+ {\bf 1}$ belong to the interior of $\Gamma_{u,\infty}$.
\end{theorem}

\begin{proof}
    Take any $\epsilon\in (0,1)$. Since $(1-\epsilon)\Psi_u+C\le u\le\Psi_u +C$ on $\Cn\setminus\D_R^n$ for some $R>1$, we have 
    \[
    \cP(\Psi_u)\subset\cP(u)\subset \cP((1-\epsilon)\Psi_u).
    \]
    By Theorem~\ref{thm:lbound}, 
    $\cP(u)\subset \cP(c\Psi)$ for any positive $c<1$.
        By Proposition~\ref{prop:multind}, we have $z^J\not\in \cP(\Psi_u)$ if and only if $J+ {\bf 1}$ does not belong to the interior of $\Gamma_{\Psi_u^+}$ and, therefore, to the interior of $(1-\epsilon)\Gamma_{\Psi_u^+}=\Gamma_{(1-\epsilon)\Psi_u^+}$ for some $\epsilon$. This gives $\cP(\Psi_u)=\cP(u)$ and, since $\Gamma_{\Psi_u^+}=\Gamma_{u,\infty}$, proves the theorem.
\end{proof}

\begin{remark}\label{rem:How}
    A local version of this result, usually called Howald's theorem, was established in \cite{Howald} (for multipliers of monomial ideals) and in \cite{Gue}, \cite{Rash13} in the psh setting. It says that a monomial $z^J$ belongs to the multiplier ideal of a toric psh function if and only if $J+\bone$ belongs to the interior of the indicator diagram $\Gamma_{u,0}$ of $u$ at $0$.
\end{remark}

For any $u\in\LPSH\Cn$, one can also consider the subspace $\wcP(u)$ of $\cP(u)$, defined by the condition 
\beq\label{eq:globmultiplier}
|P|e^{-u}\in L^2(\Cn).
\eeq
In the toric case, Theorem~\ref{theo:How2} and Remark~\ref{rem:How} imply

\begin{corollary}
    Let $u\in\TLPSHs\Cn$. Then a monomial $z^J$ belongs to $\wcP(u)$ if and only if $J+ {\bf 1}$ belongs to the interior of $\Gamma_{u,\infty}\cap \Gamma_{u,0}$.
\end{corollary}

\section{Bergman functions and analytic approximation}\label{sec:appr}

Hilbert spaces of polynomials in $\Cn$ and corresponding Bergman functions were considered by several authors, mainly with respect to compactly supported measures; see, among many others, \cite{BlLe}, \cite{BlSh},  \cite{MaVu}, \cite{ShZe1}, \cite{ShZe2}. We will use here a non-compact approach initiated in \cite{Berman}, see more details in Remark~\ref{rem:Berman}.

Given $u\in\LPSH\Cn$, let $\wcP(u)$ be the vector space of polynomials $P$ satisfying \eqref{eq:globmultiplier}. If it is nontrivial, then $\wcP(u)$ is a Hilbert space of finite dimension $N_u$, with the inner product
\[
\langle P,Q \rangle_u = \int_\Cn P \overline{Q} e^{-2u}\,\beta_n.
\]
For an orthonormal basis $\{p_l\}_1^{N_u}$ of $\wcP(u)$, the function 
\[
K_{u}(z,w)=\sum_1^{N_u} p_l(z)\overline {p_l(w)}
\]
is the reproducing kernel for the point evaluation functional $P\mapsto P(z)$:
\[
P(z)=\langle P(w),K_u(z,w) \rangle_u \quad \forall z\in\Cn. 
\]

\begin{definition}
The function 
\[
B_{u}(z):=K_{u}(z,z)=\sum_1^{N_u} |p_l(z)|^2
\]
is the {\it Bergman kernel function of $u$ in $\Cn$}.
\end{definition}

Since $B_u(z)^{1/2}$ is the norm of the evaluation functional, 
\beq\label{eq:eval}
B_u(z)=\sup\{|P(z)|:\: \|P\|_u\le 1\}.
\eeq
We put 
\beq\label{eq:um}
u_m(z)=\frac1{2m}\log B_{mu}(z),\ m>0.
\eeq

A local variant of this procedure, with integration over a bounded domain $\Omega$ of $\Cn$, was used by Demailly \cite{D92} for constructing an approximation of any function $u\in\PSH\Omega$ by functions $u_m\in\PSH\Omega$ with analytic singularities. The approximants $u_m$ were shown to converge to $u$ pointwise and in $L_{loc}^1$ because
\beq\label{eq:Dappro}
u(z)-\frac{C_1}{m} \le u_m(z)\le \sup_{\zeta\in B_r(z)}u(\zeta)
+\frac1{m}\log \frac{C_2}{ r^n}
\eeq
for some constants $C_1$ and $C_2$ independent of $m$ and $m$. They also keep control on its singularities: $\nu_a(u_m)\to\nu_a(u)$ and $c_a(u_m)\to c_a(u)$ for every $a\in \Omega$ \cite{DK}.

One can show that the right inequality in \eqref{eq:Dappro} holds as well in the global case. Indeed, for any $P\in\wcP({mu})$ with $\|P\|_{mu}\le 1$, we have, by the mean value property,
$$|P(z)|^2\le\frac1{V_nr^{2n}}\int_{\B_r(z)}|P|^2\,\beta_n\le \frac{||P||_{mu}^2}{V_nr^{2n}}\,\exp\{2m\sup_{\B_r(z)} u\},
$$
$V_n$ being the Euclidean volume of the unit ball in $\Cn$. By (\ref{eq:eval}), this gives us the second inequality in (\ref{eq:Dappro}).

The proof of the left inequality in \eqref{eq:Dappro} is based on a particular case of the Ohsawa-Takegoshi theorem on extension of holomorphic functions in weighted spaces. It says that if $u$ is a psh function in a bounded domain $\Omega$, then for any $z\in\Omega$ and any $a\in\C$ there exists a holomorphic function $f$ in $\Omega$ such that 
$f(z)=a$ and
\[
\int_\Omega |f(\zeta)|^2e^{-2u(\zeta)}\beta_n\le A \|a\|^2 e^{-2u(z)}
\]
with constant $A$ independent of $u$, $a$ and $z$; in particular, $f$ belongs to the multiplier ideal sheaf of $u$ in $\Omega$. The argument works only for bounded domains, which does not fit our setting. Instead, we are going to use its version due to Manivel \cite{Man}, according to which for any $z\in\Cn$, $a\in\C$ and $\kappa>n$ there exists a polynomial $P$ such that $P(z)=a$ and
\beq\label{eq:Man}
\int_\Cn \frac{|P(\zeta)|^2e^{-2u(\zeta)}}{(1+|\zeta|^2)^{\kappa}}
\beta_n\le A |a|^2e^{-2u(z)}
\eeq
with a constant $A$ independent of $u$, $a$ and $z$. 

Note that \eqref{eq:Man} does not imply $P\in\wcP(u)$, only that
$P\in\wcP(u+\kappa\Lambda)$ with $\Lambda(z)=\frac12 
\log(1+|z|^2)\in\LPSHs\Cn$. That is why we modify the construction of the approximants $u_m$ as follows. 

We fix any $\kappa>n$ (for example, $\kappa=n+1$) and consider the Bergman functions $B_{m}$ of the Hilbert spaces 
\[
\cH_{m}=\wcP(mu+\kappa\Lambda)
\]
(non-trivial in view of \eqref{eq:Man}) with the corresponding norm denoted for brevity by $\|\cdot,\|_m$,
and define 
\beq\label{eq:umkappa}
u_m(z)=\frac1{2m}\log B_{m}(z),\ m>0.
\eeq
For any $P\in\cH_{m}$, we have, as before,
\[
|P(z)|^2\le\frac1{V_nr^{2n}}\int_{\B_r(z)}|P|^2\,\beta_n\le \frac{||P||_{m}^2}{V_nr^{2n}}\,\exp\{2\sup_{\B_r(z)} (mu+\kappa\Lambda)\}
\]
This, by \eqref{eq:eval}, gives us
\beq\label{eq:Dkappaleft}
u_m(z)\le \sup_{\zeta\in B_r(z)}u(\zeta)+\frac1m \left[\kappa\Lambda(|z|+r) 
+\log \frac{C_2}{ r^n}\right]
\eeq
with $C_2=V_n^{-1/2}$.

To get an analog of the left inequality in \eqref{eq:Dappro}, we apply \eqref{eq:Man} to any $z\in\Cn$ such that $u(z)\neq-\infty$ with $a=A^{-1/2}e^{mu(z)}$. Then the right-hand side of \eqref{eq:Man} equals $1$ and thus, the corresponding polynomial $P_m\in\cH_{m}$ has the norm not exceeding $1$. Therefore,
\[
u_m(z)\ge \frac{\log|P_m(z)|}{m}=\frac{\log|a|}{m}=u(z)-\frac{C_1}{m}
\]
with $C_1=\frac12 \log A$.

This proves part (i) of the following theorem which is a global counterpart of Demailly's approximation theorem \cite[Proposition 3.1]{D92}, \cite[Theorem 4.2]{DK}. 

\begin{theorem}\label{thm:Deminf}
    Let $u\in\LPSH\Cn$ and $\kappa>n$. Then for any $m>0$ there exist polynomials $p_l^{(m)}$, $l=1,\ldots,N^{(m)}$, such that the function 
    \[
    u_m=\frac1{2m} \log\sum_{l=1}^{N^{(m)}}|p_l^{(m)}|^2\in\LPSH\Cn
    \]
    has the following properties:
    \begin{enumerate}
        \item[(i)] there exists constants $C_1,C_2>0$, independent of $u$ and $m$, such that 
      \beq\label{eq:Mappro}
u(z)-\frac{C_1}{m} \le u_m(z)\le \sup_{\zeta\in B_r(z)}u(\zeta)
+\frac1m \left[\kappa\Lambda(|z|+r) 
+\log \frac{C_2}{ r^n}\right]
\eeq  
for any $z\in\Cn$; in particular, $u_m$ converge to $u$ pointwise and in $L_{loc}^1(\Cn)$; 
\item[(ii)] if $u\in\LPSHs\Cn$, then $u_m\in\LPSHs\Cn$;
      \item[(iii)]  $\sigma_u\le\sigma_{u_m}\le \sigma_u+\frac1m$;
      \item[(iv)] $\LCT(u_m)\le \LCT(u)$; 
      \item[(v)] if $u$ has positive Łojasiewicz exponent, then $\LCT(u_m)\to \LCT(u)$ as $m\to\infty$.
    \end{enumerate}
    \end{theorem}

\begin{proof} Assertions (ii) and (iv) follow trivially from the first inequality in \eqref{eq:Mappro}, and it also implies the first inequality of (iii). By the second inequality in \eqref{eq:Mappro}, we get for any $r>1$,
\[
\sup_{z\in B_r(0)}u(z) \le \sup_{z\in B_{2r}(0)}u(z) + \frac1{2m}\log(1+4r^2).
\]
By dividing the both sides by $\log r$ and letting $r\to\infty$, we get the second inequality of (iii).

        To prove (v), we assume 
    \beq\label{eq:Loja} u(z)\ge 2\alpha\Lambda(z)
    \eeq
    for all $z$ with $|z|\ge R$ for some $\alpha>0$ and $R>>0$; note that, in particular, $u\in\LPSHs\Cn$.
        Take any $\epsilon>0$. By the H\"older inequality with $p=1+m/c$ and $q=1+c/m$   with any $c\in(\LCT(u_m)+\epsilon, \LCT(u_m)+2\epsilon)$, we have 
\begin{eqnarray*}
\int_{\Cn\setminus\B_R} e^{-\frac{2mu}{p}} \beta_n &=& \int_{\Cn\setminus\B_R} \left[\frac{e^{-2mu}B_m}{(1+|z|^2)^\kappa}\right]^{\frac1p}
\left[\frac{(1+|z|^2)^{\kappa}}{B_m}\right]^{\frac{c}{qm}}\beta_n\\
&\le &  C_m \left(\int_{\Cn\setminus\B_R}
(1+|z|^2)^{\frac{\kappa c}{m}}e^{-2cu_m}\beta_n\right)^{\frac1q}\\
&\le & C_m  \left(\int_{\Cn\setminus\B_R}
(1+|z|^2)^{\frac{\kappa c}{m}-2\epsilon\alpha}e^{-2(c-\epsilon)u_m}\beta_n\right)^{\frac1q}
\end{eqnarray*}
where $C_m=(N^{(m)}+1)^{\frac1p}$, $N_m$ being the dimension of $\cH_m$; the last inequality comes from \eqref{eq:Loja}.

Since $\LCT(u)\ge \LCT(u_m)$, we have $\kappa c/m<\epsilon\alpha$
for any 
\[
m>\frac{\kappa (\LCT(u)+2\epsilon)}{2\epsilon\alpha}.
\]
For all such $m$, the last integral is less than
\[
\int_{\Cn\setminus\B_R}
e^{-2(c-\epsilon)u_m}\beta_n,
\]
 which is finite because $c-\epsilon>\LCT(u_m)$. Therefore,   
 \[
 1/{\LCT(u)}=\lambda_\infty(u)\ge \frac{p}{m}=\frac1c+\frac1m\ge \frac1{\LCT(u_m)+2\epsilon}+\frac1m
 \]
 and so, 
 \[
 \liminf_{m\to\infty}\LCT(u_m)\ge \LCT(u)-2\epsilon,
 \]
 which, in view of (iv), implies $\LCT(u_m)\to \LCT(u)$.
 \end{proof}

\begin{remark}
    In the proof of (iv), one can choose, for any $s\in(0,1)$,
    $\epsilon =\epsilon_{m,s}=2m^{-s}$ and get the explicit bound
    \[
    \frac1{\LCT(u_m)+m^{-s}}\le \frac1{\LCT(u)}-\frac1m
    \]
for all $m$ such that $m^s>\kappa \LCT(u)/\alpha$, where $\alpha>0$ is the value from \eqref{eq:Loja}; cf. \eqref{eq:Dapprolct} below for the local case.
\end{remark}

\begin{remark}
    Inequalities \eqref{eq:Mappro} imply, as in Demailly's approximation theorem on bounded domains, that for any $a\in\Cn$,
   \[ \nu_a(u)-\frac{n}{ m}\le\nu_a(u_m)\le\nu_a(u) 
   \]
and
\beq\label{eq:Dapprolct}
\frac1{c_a(u)}-\frac1{m}\le \frac1{c_a(u_m)}\le \frac1{c_a(u)},
\eeq
with the same proofs as in \cite{D92} and \cite{DK}.
\end{remark}

\begin{remark}\label{rem:Berman}
    A procedure of constructing a Bergman function for weighted polynomial spaces in $\Cn$ with the weights $e^{-m\phi}$  was suggested in \cite{Berman}. There, the function $\phi$ was assumed to satisfy the growth condition $\phi\ge (1+\epsilon)\Lambda$ near infinity and to be $C^{1,1}$-smooth in $\Cn$. The main difference, however, with our setting is that the Hilbert spaces $\cH_m$ in \cite{Berman} -- as in other related papers -- were formed by polynomials $P$ satisfying, in addition to the integrability condition, the restriction $\deg P< m$. This resulted in the conclusion 
   $ u_m\le \Lambda + O(1)$
   for the corresponding functions $u_m$ 
        and in their convergence to a function $\phi_e$,  the largest psh function not exceeding $\min\{\phi,\Lambda + O(1)\}$.
    
\end{remark}

\section{Log canonical thresholds and Monge-Ampère masses}

If the  Monge-Ampère operator $(dd^cu)^n$ is well defined near a point $\zeta\in\Cn$, the log canonical threshold $c_\zeta(u)$ is nicely related to the residual Monge-Ampère mass 
$\tau_a(u)=(dd^cu)^n(\{\zeta\})$
of $u$ at $\zeta$: 
\beq\label{eq:dFEMD}
c_\zeta(u)\ge n\,\tau_\zeta(u)^{-\frac1n}.
\eeq
The inequality was originally proved in \cite{dFEM1}, \cite{dFEM2} in the case of ideals of holomorphic germs and then extended in \cite{D09} to psh functions with isolated singularities (actually, the formula relating the Samuel multiplicity of the local ideal $\cI_\zeta(F)$ with the residual Monge-Ampère mass of $(dd^c\log|F|)^n$ at $\zeta$, mentioned in Section~\ref{sect:pshpol}, was proved in \cite{D09} exactly for this extension). 
Later on, the inequality was proved without referring to the case of ideals \cite{ACKPZ}, and a stronger relation was established in \cite{DP}. In \cite{Rash15}, it was shown that equality in \eqref{eq:dFEMD} occurs if and only if $\Psi_{u,\zeta}(z)=S_{a}(z-\zeta)$ \eqref{eq:phia} for $a=\alpha\bone$ with some $\alpha>0$.

In the global setting of the class $\LPSHs\Cn$, the counterpart of the residual mass at a point is naturally the {\it total Monge-Ampère mass} 
\[
\m_n(u):=\int_\Cn (dd^c u)^n
\]
or, more generally, the {\it total mixed Monge-Ampère masses} 
   $$
    \m_k(u)=\int_\Cn (dd^c u)^k\wedge (dd^c\log|z|)^{n-k},\quad 0\le k\le n.
    $$
    In particular, and $\m_0(u)=\m_n(\log|z|)=1$. (Note that the Monge-Ampère operator is well-defined on all functions $u\in\LPSHs\Cn$ because $u^{-1}(-\infty)\Subset\Cn$.)

\medskip

   It follows from a mixed products version of B.A Taylor's comparison theorem) \cite[Theorem 3.1]{Rash03} that if 
   $u_k,v_k\in\LPSHs\Cn$, $1\le k\le n$, are such that for any $\eta>0$,
        \[
        \limsup_{|z|\to\infty}\frac{u_k(z)}{v_k(z)+\eta\log|z|}\le l_k<\infty,
        \]
        then 
        \[
        \int_\Cn dd^cu_1\wedge\ldots\wedge dd^cu_n\le
        l_1\cdot\ldots\cdot l_n\int_\Cn dd^cv_1\wedge\ldots\wedge dd^cv_n<\infty.
        \]
   This gives us immediately $\m_n(u)=\m_n(u^+)$ and, since $u\le\Psi_u+O(1)$, the relations $\m_n(u)\le\m_n(\Psi_u)=\m_n(\Psi_u^+)$ for any $u\in\LPSHs\Cn$ and, more generally,
    \beq\label{eq:totalMA}
    \m_k(u)\le\m_k(\Psi_u)=\m_k(\Psi_u^+), \quad 1\le k\le n,
    \eeq
see \cite[Prop. 4.1]{Rash03}.

\begin{remark}
    The relation $\m_n(u)=\m_n(u^+)$ is not in general true for $u\in\LPSH\Cn\setminus\LPSHs\Cn$ with well-defined Monge-Ampère measure. For example, if $u(z_1,z_2)=\log(|z_1|^2+|z_1z_2+1|^2)$, then $\m_2(u)=0<4=\m_2(u^+)$, see \cite[p. 174]{Rash03}.
    \end{remark}

If $u\in\TLPSHs\Cn$, inequality \eqref{eq:lbound} gives us
     $\m_k(u)\ge(1-2\epsilon)^k\,\m_k(\Psi_u^+)$ and so, since $\epsilon>0$ is arbitrary, $\m_k(u)\ge\m_k(\Psi_u)$. Combining it with the reverse inequality \eqref{eq:totalMA}, we get

\begin{theorem}\label{thm:fullmasses}
    For any $u\in\TLPSHs\Cn$, we have $$\m_k(u)=\m_k(\Psi_u), \quad 1\le k\le n.$$
\end{theorem}

This gives an efficient way of computing the total masses of the toric functions because, by \cite{Rash01}, 
\beq\label{eq:masspsi}
\m_n(\Psi_u)=n!\,{\rm Vol}(\Gamma_{u,\infty}),
\eeq
where $\Gamma_{u,\infty}=\Gamma_{\Psi_u^+}$ is the indicator diagram of $u$ at infinity. Similarly, $\m_k(u)$ equals $n!$ times the mixed volume of $k$ copies of $\Gamma_{u,\infty}$ and $n-k$ copies of the standard simplex $\Delta=\{a\in{\overline\Rnp}:\: \sum_j a_j\le 1\}$ \cite{Rash03}.

In particular, by Proposition~\ref{prop:globDem}, this gives a formula for the global multiplicity $\m(P)$ (see Definition~\ref{def:globmult}) of monomial maps $P\in\cPs$. Moreover, using arguments of the proof of Corollary~\ref{cor:NNDlct}, including \eqref{eq:spcl},  we deduce

\begin{corollary}
   Let $P:\C^n\to \C^q$ be a polynomial map, convenient and Newton
non-degenerate at infinity. Then 
\[
\m(P)=n!\,{\rm Vol}(\Gamma_{P,\infty}).
\]
\end{corollary}

\medskip

We refer to Corollary 3.9 of \cite{BH} for an alternative proof of the above equality
and the characterization of the polynomial maps for which this equality holds. 

\medskip
Comparing with the local case \eqref{eq:dFEMD}, one may ask if there is a relation also between $\m_n(u)$ and $\LCT(u)$ for all $u\in\LPSHs\Cn$. 
It turns out that this is not generally the case, even for monomial maps.

\begin{example}\label{ex:antiD} For any $N\ge 2$,
    take 
    \[
    P_N(z_1,z_2)=(z_1,z_2,z_1^Nz_2^N, 1) \ {\rm and\ }
        Q_N(z_1,z_2)=(z_1,z_2^N,z_1z_2^N,1).
\]
For the functions $u_N=\frac1N\log|P_N|$ and $v_N=\log|Q_N|$ from $\LPSHs{\C^2}$, we have $\LCT(u_N)=\LCT(v_N)=1$ while $\m_n(u_N)=2/N$ and $\m_n(v_N)=2N$.
    \end{example}

However, there is a relation between $\m_n(u)$ and $\LCT(u)$ for $u\in\TLPSHs{\Cn}$, provided the indicator diagram $\Gamma_{u,\infty}$ is a lower set, see Definition~\ref{def:lower}. (Note that the complement to the indicator diagram $\Gamma_{u,0}$ is always a lower set and the indicator diagram $\Gamma_{u,0}$ itself is always an {\sl upper set}.)

\begin{proposition} \label{prop:FMDglob}
    For any $u\in\TLPSHs{\Cn}$ whose indicator diagram $\Gamma_{u,\infty}$ is  a lower set, we have
    \beq\label{eq:Dglob}
\LCT(u)\ge \left(\frac{n!}{\m_n(u)}\right)^{\frac1n}.
\eeq
Then inequality becomes equality if and only if  $\Gamma_{u,\infty}$ is a cube $[0,\alpha]^n$ (in which case $\LCT(u)=\alpha^{-1}$).
\end{proposition}

\begin{proof}     Take any $c>\LCT(u)$. By Theorem~\ref{thm:toric_Kiselman}, the point $c^{-1}\bf 1$ belongs to the interior of the indicator diagram $\Gamma_{u,\infty}$. Since $\Gamma_{u,\infty}$ is a lower set, we have $[0, c^{-1}]^n\subset\Gamma_{u,\infty}$, so ${\rm Vol}(\Gamma_{u,\infty})\ge c^{-n}$.
    By Theorem~\ref{thm:fullmasses} and \eqref{eq:masspsi}, $\m_n(u)=n!\,{\rm Vol}(\Gamma_{u,\infty})$. Taking the infimum over all $c>\LCT(u)$, we get \eqref{eq:Dglob}. Finally, we have equality in \eqref{eq:Dglob} if and only if ${\rm Vol}(\Gamma_{u,\infty})= c^{-n}$, which completes the proof.
\end{proof}

\begin{remark}
    Note, that unlike the local situation (see the first paragraph of this section), the extremal case is not the equilateral simplicial functions $S_{\alpha\bone}$, for which we have, by \eqref{eq:lct phia}-\eqref{eq:Sa},
    \[
    \LCT(S_{\alpha\bone})[\m_n(S_{\alpha\bone})]^{1/n}=n > (n!)^{1/n}.
    \] 
\end{remark}

By using arguments of the proof of Corollary~\ref{cor:NNDlct}, we get

\begin{corollary}   Let $P:\C^n\to \C^q$ be a polynomial map, convenient and Newton
non-degenerate at infinity. If its Newton polyhedron at $\infty$ is a lower set, then 
\[
\LCT(P)\ge \left(\frac{n!}{\m(P)}\right)^{\frac1n}.
\]

\end{corollary}

\begin{remark}
    An easy example showing that the lower condition is needed is the map $P:\C^2\to \C^4$ given by $P(z_1,z_2)=(z_1^3z_2^3, z_1^3, z_1z_2^3, z_2^2)$. In this case, $\m(P)=17$ and the bound (approx. 0.34) is greater than $\LCT(P)=1/3$.
\end{remark}

\begin{remark}
One might also look for refined relations involving $\m_k(u)$ in the spirit of the Demailly-Pham inequality \cite{DP}.
\end{remark}



\bigskip

\begin{thebibliography}{AA}

\bibitem{ACKPZ}
P.\,Åhag, U.\,Cegrell, S.\,Ko\l odziej, Pham Hoang Hiep, A.\,Zeriahi, 
\newblock{\it Partial pluricomplex energy and
integrability exponents of plurisubharmonic functions}, 
\newblock Adv. Math. {\bf 222} (2009), 2036–2058.

\bibitem{AYu}  L.\,A.\,Aizenberg and Yu.\,P.\,Yuzhakov,
\newblock{Integral Representations and Residues in Multidimensional
Complex Analysis.}.
\newblock Nauka, Novosibirsk, 1979. English transl.:
AMS, Providence, R.I., 1983..

\bibitem{BBL} T.\,Bayraktar, T.\,Bloom and N.\,Levenberg,
\newblock{\it Pluripotential theory and convex bodies},
\newblock Sb. Math. {\bf 209} (2018), No. 3, 352-384; translation from Mat. Sb. 209, No. 3, 67-101 (2018).

\bibitem{Berman}
R.\,Berman, 
\newblock{\it Bergman kernels for weighted polynomials and weighted equilibrium measures of $\Cn$}, 
\newblock Indiana Univ. Math. J. {\bf 58} (2009), No. 4, 1921--1946.

\bibitem{BB}
R.\,Berndtsson, 
\newblock{\it The openness conjecture for plurisubharmonic functions}, 
\newblock preprint at arXiv:1305.5781.


\bibitem{B} C.\,Bivià-Ausina,
\newblock{\it Log canonical threshold and diagonal ideals},
\newblock Proc. Amer. Math. Soc. {\bf 145}, no. 5 (2017), 1905–1916.



\bibitem{BF1} C.\,Bivià-Ausina and T.\,Fukui,
\newblock{\it Mixed \L ojasiewicz exponents and log canonical thresholds of ideals},
\newblock J. Pure Appl. Algebra  {\bf 220}, no. 1 (2016), 223–245.

\bibitem{BH17} C.\,Bivià-Ausina and J.\,A.\,C.\,Huarcaya,
\newblock{\it The special closure of polynomial maps and global non-degeneracy}, 
\newblock Mediterr. J. Math. {\bf 14}, No. 2 (2017), Art. 71, 21 pp.

\bibitem{BH} C.\,Bivià-Ausina and J.\,A.\,C.\,Huarcaya,
\newblock{\it Global multiplicity, special closure and non-degeneracy of gradient maps},
\newblock Rev. R. Acad. Cienc. Exactas Fis. Nat. Ser. A Mat.  {\bf 117:121} (2023).

\bibitem{BL} M.\,Blickle and R.\,Lazarsfeld,
\newblock{\it An informal introduction to multiplier ideals},
\newblock in Trends in Commutative Algebra, 87–114, Math. Sci. Res. Inst. Publ., 51, Cambridge Univ. Press, Cambridge, 2004.


\bibitem{BlLe}
T.\,Bloom and N.\,Levenberg, 
\newblock{\it Random polynomials and pluripotential-theoretic extremal functions}, 
\newblock Potential Anal. {\bf 42} (2015), 311–334.

\bibitem{BlSh}
T.\,Bloom and B.\,Shiffman, 
\newblock{\it Zeros of random polynomials on $\C^m$}, 
\newblock Math. Res. Lett. {\bf 14} (2007), 469–479.

\bibitem{BoLe}
L.\,Bos and N.\,Levenberg, 
\newblock{\it Bernstein-Walsh theory associated to convex bodies and applications to multivariate approximation theory}, 
\newblock Comput. Methods Funct. Theory {\bf 18} (2018), 361–388.

\bibitem{CLS} D.\,Cox, J.\,Little and D.\, O'Shea,
\newblock Using Algebraic Geometry. 
\newblock Graduate Texts in Mathematics, vol. 185. Springer, Berlin (2005).


\bibitem{Dbook}
J.-P.\,Demailly, 
\newblock{Complex Analytic and Differential Geometry.} 
\newblock Available at
\url{http://www-fourier.ujf-grenoble.fr/~demailly/books.html}.

\bibitem{D92}
J.-P.\,Demailly, 
\newblock{\it Regularization of closed positive currents and intersection theory}, 
\newblock J. Algebraic Geometry {\bf 1} (1992), 361--409.



\bibitem{D09}
J.-P.\,Demailly, 
\newblock{\it Estimates on Monge-Ampère operators derived from a local algebra inequality}, 
\newblock in:
Complex Analysis and Digital geometry, Proceedings of the Kiselmanfest 2006, Acta Universitatis Upsaliensis, 2009.


\bibitem{DK} J.\,P.\,Demailly and J.\,Koll\'ar,
\newblock{\it Semi-continuity of
complex singularity exponents and K\"ahler-Einstein metrics on
Fano orbifolds},
\newblock  Ann. Sci. Ecole Norm. Sup. (4) {\bf 34} (2001),
no. 4, 525--556.



\bibitem{DP} J.\,P.\,Demailly and H.\,H.\,Pham,
\newblock{\it A sharp lower bound for the log canonical threshold},
\newblock Acta Math. {\bf 212} (2014), no. 1, 1–9.


\bibitem{Ein2004} L.\,Ein, R.\,Lazarsfeld, K.\,E.\,Smith and D.\,Varolin,
\newblock{\it Jumping coefficients of multiplier ideals},
\newblock Duke Math. J. {\bf 123}, no. 3 (2003), 469–506.

\bibitem{dFEM1} T.\,de\,Fernex, L.\,Ein and M.\,Musta\c{t}\u a,
\newblock{\it Multiplicities and log canonical threshold},
\newblock J. Algebraic Geom. {\bf 13} (2004), no. 3, 603--615.

\bibitem{dFEM2} T.\,de\,Fernex and M.\,Musta\c{t}\u a,,
\newblock{\it Limits of log canonical thresholds},
\newblock Ann. Sci. Éc. Norm. Supér. (4) 42 (2009), No. 3, 491--515.

\bibitem{GZ}
{ V.\,Guedj and A.\,Zeriahi}, 
\newblock Degenerate complex Monge-Ampère equations.
\newblock EMS Tracts in Mathematics 26. Z\"{u}rich: European Mathematical Society (EMS), 2017.  

\bibitem{Gue} H.\,Guenancia,
\newblock{\it Toric plurisubharmonic functions and analytic adjoint ideal sheaves},
\newblock Math. Z. {\bf 271} (2012), No. 3-4, 1011--1035.

\bibitem{HCJ} W.\,K.\,Hayman, P.\,M.\,Cohn, 
B.\,E.\,Johnson,
\newblock{\it Subharmonic functions, vol.2}.
\newblock (London Mathematical Society Monographs), Academic Press (1990).


\bibitem{Howald} J.\,Howald,
\newblock{\it Multiplier ideals of monomial ideals},
\newblock Trans. Amer. Math. Soc. {\bf 353} (2001), no. 7, 2665--2671

\bibitem{Kiselman} C.\,O.\,Kiselman,
\newblock{\it Attenuating the singularities of plurisubharmonic functions},
\newblock Ann. Pol. Math. {\bf 60}, No.2, 173-197 (1994).

\bibitem{Klimek}
M.\,Klimek,
\newblock{Pluripotential Theory}.
\newblock Clarendon Press, 1991.


\bibitem{Kollar} J.\,Kollár,
\newblock{\it Which powers of holomorphic functions are integrable?},
\newblock ArXiv:0805.0756v1 [math.AG]


\bibitem{Ku}
{A.\,G.\,Kouchnirenko}, 
\newblock{\it Poly\`edres de Newton et
nombres de Milnor}, 
\newblock Invent. Math. {\bf 32} (1976), 1--31.

\bibitem{Krasinski}
{T.\,Krasi\'nski}, 
\newblock{\it On the \L ojasiewicz exponent at infinity of polynomial mappings}, 
\newblock Acta Math. Vietnam. {\bf 32} (2007), no. 2-3, 189–203. 

\bibitem{Levin} B.Ya.\,Levin,
\newblock{\it Lectures on Entire Functions}.
\newblock Translations of Mathematical Monographs, Vol. 150; AMS, 1996.

 \bibitem{MSSS}
B.\,S.\,Magnússon, Á.\,E.\,Sigurðardóttir, R.\,Sigurðsson, and B.\,Snorrason, 
\newblock{\it Polynomials with exponents in compact convex sets and associated weighted extremal functions - Fundamental results}, 
\newblock Ann. Polon. Math. {\bf 133} (2024), no. 1, 37--70.

\bibitem{Man} L.\,Manivel,
\newblock{\it Un théorème de prolongement $L^2$ de sections holomorphes d'une fibre hermitien},
\newblock Math. Z. {\bf 212} (1993), no. 1, 107--122.

\bibitem{MaVu}
G.\,Marinescu and D.-V.\,Vu, 
\newblock{\it Bergman kernel functions associated to measures supported on totally real manifolds}, 
\newblock J. Reine Angew. Math. {\bf 810} (2024), 217–251.

\bibitem{McZ} J.\,D.\,McNeal and Y.\,E.\,Zeytuncu,
\newblock{\it Multiplier ideals and integral closure of monomial ideals: an analytic approach.},
\newblock Proc. Amer. Math. Soc. {\bf 140} (2012), no. 5, 1483–1493.

\bibitem{Mu1} M.\,Musta\c{t}\u a,
\newblock{\it On multiplicities of graded sequences of ideals},
\newblock Journal of Algebra {\bf 256} (2002), 229–249.

\bibitem{Mu2} M.\,Musta\c{t}\u a,
\newblock{\it \textsc{impanga} lecture notes on log canonical thresholds},
\newblock notes by Tomasz Szemberg. EMS Ser. Congr. Rep., Contributions to algebraic geometry, 407–442, Eur. Math. Soc., Zürich, 2012.

\bibitem{Rash00}
{A.\, Rashkovskii}, 
\newblock{\it Newton numbers and residual measures of
  plurisubharmonic functions}, 
  \newblock Ann. Polon. Math. {\bf 75} (2000),
no. 3, 213--231.

\bibitem{Rash01}
{A.\,Rashkovskii}, 
\newblock{\it Indicators for plurisubharmonic functions
of logarithmic growth}, 
\newblock Indiana Univ. Math. J. {\bf 50} (2001),
no.~3, 1433--1446.

\bibitem{Rash03}
{A.\,Rashkovskii}, 
\newblock{\it Total masses of mixed Monge-Amp\`ere
currents}, 
\newblock Michigan Math. J. {\bf 51} (2003), no.~1, 169--186.

\bibitem{Rash13} A.\,Rashkovskii,
\newblock{\it Multi-circled singularities, Lelong numbers and integrability index},
\newblock J. Geom. Anal. {\bf 23} (2013), no. 4, 1976–-1992.

\bibitem{Rash15}
{A. Rashkovskii}, 
\newblock{\it Extremal cases for the log canonical threshold}, 
\newblock C. R. Acad. Sci. Paris, Ser. I {\bf 353} (2015), 21--24.

\bibitem{Rash17} A.\,Rashkovskii,
\newblock{\it A log-canonical threshold test},
\newblock Analysis Meets Geometry: A Tribut to Mikael Passare. Trends in Mathematics, 361-368. Birkhäuser, 2017.

\bibitem{SW} A.\,J.\, Sommese and S.\,W.\, Wampler,
\newblock The Numerical Solution of Systems of Polynomials Arising in Engineering and Science.
\newblock World Scientific (2005).

\bibitem{STot} E.\,Saff and V.\,Totik,
\newblock Logarithmic Potentials with Exterior Fields (with and appendix by T.\, Bloom).
\newblock Springer-Verlag (1997), Berlin.


\bibitem{ShZe1}
B.\,Shiffman and S.\,Zelditch, 
\newblock{\it Equilibrium distribution of zeros of random polynomials}, 
\newblock Int. Math. Res. Not. {\bf 1} (2003), 25–49.

\bibitem{ShZe2}
B.\,Shiffman and S.\,Zelditch, 
\newblock{\it Random polynomials with prescribed Newton polytope}, 
\newblock J. Amer. Math. Soc. {\bf 17} (2004), 49–108.


\bibitem{Si96} J.\,Siciak,
\newblock{ A remark on Tchebysheff polynomials in $\Cn$}.
\newblock Preprint, Jagellionian University, 1996.

\bibitem{Sk} {H.\,Skoda}, 
\newblock{\it Sous-ensembles analytiques d'ordre fini ou
infini dans $\Cn$}, 
\newblock Bull. Soc. Math. France {\bf 100} (1972),
353--408.

\bibitem{Sno} B.\,Snorrason, 
\newblock{\it Polynomials with exponents in compact convex sets and associated weighted extremal functions - Approximations and regularity}, 
\newblock arXiv:2410.20370.



\end{thebibliography}
\end{document}